\documentclass[11pt]{article}
\usepackage{amsfonts}
\usepackage{latexsym}
\usepackage{amsmath}
\usepackage{amssymb}
\usepackage{amsthm}
\usepackage{dsfont}
\newtheorem{theorem}{Theorem}[section]
\newtheorem{lemma}[theorem]{Lemma}

\newtheorem{proposition}[theorem]{Proposition}
\newtheorem{corollary}[theorem]{Corollary}

\theoremstyle{definition}

\theoremstyle{remark}
\newtheorem*{remark}{Remark}

\newtheorem*{note*}{Note}

\numberwithin{equation}{section}

\makeatletter
\newcommand{\ls}{\leqslant}
\newcommand{\gr}{\geqslant}

\newcommand{\R}{\mathbb R}
\newcommand{\N}{\mathbb N}

\newcommand{\eps}{\varepsilon}
\newcommand{\EE}{\mathbb{E}}
\makeatother

\newcommand{\wtd}[1]{\widetilde{#1}}
\newcommand{\abs}[1]{\left\lvert#1\right\rvert}
\newcommand{\norm}[1]{\left\lVert#1\right\rVert}

\newcommand{\conv}{{\rm conv}}

\begin{document}
%\small

\title{Randomized isoperimetric inequalities}

%\author{Grigoris Paouris \and Peter Pivovarov} 
\author{Grigoris Paouris\thanks{Grigoris Paouris is supported by US
    NSF grant CAREER-1151711 and BSF grant 2010288. } \and Peter
  Pivovarov\thanks{ This work was partially supported by a grant from
    the Simons Foundation (\#317733 to Peter Pivovarov).}}

\maketitle \abstract{
  We discuss isoperimetric inequalities for convex sets.  These
  include the classical isoperimetric inequality and that of
  Brunn-Minkowski, Blaschke-Santal\'{o}, Busemann-Petty and their
  various extensions.  We show that many such inequalities admit
  stronger randomized forms in the following sense: for natural
  families of associated random convex sets one has stochastic
  dominance for various functionals such as volume, surface area, mean
  width and others. By laws of large numbers, these randomized
  versions recover the classical inequalities.  We give an overview of
  when such stochastic dominance arises and its applications in convex
  geometry and probability.}

\section{Introduction}

The focus of this paper is stochastic forms of isoperimetric
inequalities for convex sets. To set the stage, we begin with two
examples.  Among the most fundamental isoperimetric inequalities is
the Brunn-Minkowski inequality for the volume $V_n$ of convex bodies
$K,L\subseteq \R^n$,
\begin{equation}
  \label{eqn:BM}
  V_n(K+L)^{1/n} \gr V_n(K)^{1/n} +V_n(L)^{1/n},
\end{equation}
where $K+L$ is the Minkowski sum $\{x+y:x\in K, y\in L\}$.  The
Brunn-Minkowski inequality is the cornerstone of the Brunn-Minkowski
theory and its reach extends well beyond convex geometry; see
Schneider's monograph \cite{Schneider_book_ed2} and Gardner's survey
\cite{Gardner_survey}. It is well-known that \eqref{eqn:BM} provides a
direct route to the classical isoperimetric inequality relating
surface area $S$ and volume,
\begin{equation}
  \label{eqn:iso}
  \left(\frac{S(K)}{S(B)}\right)^{1/(n-1)}\gr
  \left(\frac{V_n(K)}{V_n(B)}\right)^{1/n},
\end{equation}
where $B$ is the Euclidean unit ball.  As equality holds in
\eqref{eqn:BM} if $K$ and $L$ are homothetic, it can be equivalently
stated in isoperimetric form as follows:
\begin{equation}
  \label{eqn:BM_norm}
  V_n(K+L) \gr V_n(r_KB+ r_LB),
\end{equation}
where $r_K, r_L$ denote the radii of Euclidean balls with the same
volume as $K, L$, respectively, i.e., $r_K = (V_n(K)/V_n(B))^{1/n}$; for
subsequent reference, with this notation, \eqref{eqn:iso} reads
\begin{equation}
\label{eqn:iso_norm}
 S(K)\gr S(r_KB).
\end{equation}

Both \eqref{eqn:BM} and \eqref{eqn:iso} admit stronger empirical
versions associated with random convex sets. Specifically, let
$x_1,\ldots,x_N$ be independent random vectors (on some probability
space $(\Omega, \mathcal{F},\mathbb{P})$) distributed according to the
uniform density on a convex body $K\subseteq \R^n$, say,
$f_K=\frac{1}{V_n(K)}\mathds{1}_K$, i.e., $\mathbb{P}(x_i\in A) =
\int_{A} f_K(x)dx$ for Borel sets $A\subseteq \R^n$. For each such $K$
and $N>n$, we associate a random polytope
\begin{equation*}K_N =
  \conv\{x_1,\ldots,x_N\},
\end{equation*}
where $\conv$ denotes convex hull. Then the following stochastic
dominance holds for the random polytopes $K_{N_1}$, $L_{N_2}$ and
$(r_K B)_{N_1}$, $(r_LB)_{N_2}$ associated with the bodies in
\eqref{eqn:BM_norm}: for all $\alpha\gr 0$,
\begin{equation}
  \label{eqn:BM_sd}
  \mathbb{P}\left(V_n(K_{N_1} +L_{N_2}) >\alpha\right) \gr
  \mathbb{P}\left(V_n((r_KB)_{N_1}+(r_LB)_{N_2}) >\alpha\right).
\end{equation}
Integrating in $\alpha$ gives 
\begin{equation*}
  \EE V_n(K_{N_1} +L_{N_2}) \gr \EE V_n((r_KB)_{N_1} + (r_LB)_{N_2}),
\end{equation*}
where $\EE$ denotes expectation. By the law of large numbers, when
$N_1, N_2\rightarrow \infty$, the latter convex hulls converge to
their ambient bodies and this leads to \eqref{eqn:BM_norm}. Thus
\eqref{eqn:BM} is a global inequality which can be proved by a random
approximation procedure in which stochastic dominance holds at each
stage; for a different stochastic form of \eqref{eqn:BM}, see Vitale's
work \cite{Vitale}. For the classical isoperimetric inequality, one
has the following distributional inequality, for $\alpha \gr 0$,
\begin{equation}
  \label{eqn:iso_sd}
  \mathbb{P}\left(S(K_{N_1})>\alpha\right) \gr
  \mathbb{P}\left(S((r_KB)_{N_1})>\alpha\right).
\end{equation} 
The same integration and limiting procedure lead to
\eqref{eqn:iso_norm}. For fixed $N_1$ and $N_2$, the sets in the
extremizing probabilities on the right-hand sides of \eqref{eqn:BM_sd}
and \eqref{eqn:iso_sd} are not Euclidean balls, but rather sets that
one generates using Euclidean balls. In particular, the stochastic
forms are strictly stronger than the global inequalities
\eqref{eqn:BM} and \eqref{eqn:iso}.

The goal of this paper is to give an overview of related stochastic
forms of isoperimetric inequalities. Both \eqref{eqn:BM} and
\eqref{eqn:iso} hold for non-convex sets but we focus on stochastic
dominance associated with convex sets. The underlying randomness,
however, will not be limited to uniform distributions on convex bodies
but will involve continuous distributions on $\R^n$.  We will discuss
a streamlined approach that yields stochastic dominance in a variety
of inequalities in convex geometry and their applications.  We pay
particular attention to high-dimensional probability distributions and
associated structures, e.g., random convex sets and matrices. Many of
the results we discuss are from a series of papers
\cite{PaoPiv_probtake}, \cite{PaoPiv_smallball}, along with
D. Cordero-Erausquin, M. Fradelizi \cite{CEFPP}, S. Dann \cite{DPP}
and G. Livshyts \cite{LPP}. We also present a few new results that fit
in this framework and have not appeared previously.

%We do not know all equality cases in the stochastic forms.

Inequalities for the volume of random convex hulls in stochastic
geometry have a rich history starting with Blaschke's resolution of
Sylvester's famous four-point problem in the plane (see, e.g.,
\cite{Pfiefer_Sylvester}, \cite{CCG_Sylvester}, \cite{CG},
\cite{Gardner_book} for background and history). In particular, for
planar convex bodies Blaschke proved that the random triangle $K_{3}$
(notation as above) satisfies
\begin{equation} 
  \label{eqn:Blaschke}
\EE V_2(\Delta_3)\gr \EE V_2(K_{3}) \gr \EE V_2((r_KB_2)_{3}),
\end{equation}
where $\Delta$ is a triangle in $\R^2$ with the same area as $K$ and
$B_2$ is the unit disk. Blaschke's proof of the lower bound draws on
Steiner symmetrization, which is the basis for many related extremal
inequalities, see, e.g,. \cite{Schneider_book_ed2},
\cite{Gardner_book}, \cite{Gruber}. More generally, shadow systems as
put forth by Rogers and Shephard \cite{Shephard}, \cite{RS} and
developed by Campi and Gronchi, among others, play a fundamental role,
e.g., \cite{CCG_Sylvester}, \cite{CG_volume_product},
\cite{CG_survey}, and will be defined and discussed further below.
Finding maximizers in \eqref{eqn:Blaschke} for $n\gr 3$ has proved
more difficult and is connected to the famous slicing problem, which
we will not discuss here (see \cite{BGVV} for background).

A seminal result building on the lower bound in \eqref{eqn:Blaschke}
is Busemann's random simplex inequality \cite{Busemann},
\cite{BusemannStraus}: for a convex body $K\subseteq\R^n$ and $p\gr
1$, the set $K_{o,n} =\conv\{o,x_1,\ldots,x_n\}$ ($x_i$'s as above)
satisfies
\begin{equation}
  \label{eqn:Busemann}
  \EE V_n(K_{o,n})^p \gr \EE V_n((r_KB)_{o,n})^p.
\end{equation}
This is a key ingredient in Busemann's intersection inequality,
\begin{equation}
  \label{eqn:Busemann_int}
  \int_{S^{n-1}} V_{n-1}(K\cap \theta^{\perp})^{n}d\sigma(\theta)
  \leq \int_{S^{n-1}} V_{n-1}((r_KB)\cap
  \theta^{\perp})^{n}d\sigma(\theta),
\end{equation}
where $S^{n-1}$ is the unit sphere equipped with the Haar probability
measure $\sigma$; \eqref{eqn:Busemann} is also the basis for extending
\eqref{eqn:Busemann_int} to lower dimensional secitons as proved by
Busemann and Straus \cite{BusemannStraus} and Grinberg
\cite{Grinberg}; see also Gardner \cite{Gardner_dual} for further
extensions.

Inextricably linked to Busemann's random simplex inequality is the
Busemann-Petty centroid inequality, proved by Petty \cite{Petty}.  The
centroid body of a star body $K\subseteq \R^n$ is the convex body
$Z(K)$ with support function given by
\begin{equation*}
  h(Z(K), y) = \frac{1}{V_n(K)}\int_K \abs{\langle x, y \rangle}dx;
\end{equation*}
(star bodies and support functions are defined in \S
\ref{section:prelim}) and it satisfies
\begin{equation*}
V_n(Z(K))\gr V_n(Z(r_KB)).
\end{equation*}
The latter occupies a special role in the theory of affine
isoperimetric inequalities; see Lutwak's survey
\cite{Lutwak_survey}. 

One can view \eqref{eqn:Busemann} as a result about convex hulls or
about the random parallelotope $\sum_{i=1}^n[-x_i,x_i]$ (since
$n!V_n(K_{o,n}) = \abs{\det[x_1,\ldots,x_n]}$). Both viewpoints
generalize: for convex hulls $K_N$ with $N>n$, this was done by
Groemer \cite{Groemer} and for Minkowski sums of $N\gr n$ random line
segments by Bourgain, Meyer, Milman and Pajor \cite{BMMP}; these are
stated in \S \ref{section:iso}, where we discuss various extensions
for different functionals and underlying randomness. These are the
starting point for a systematic study of many related quantities.

In particular, convex hulls and zonotopes are natural endpoint
families of sets in $L_p$-Brunn-Minkowski theory and its recent
extensions. In the last twenty years, this area has seen significant
developments. $L_p$ analogues of centroid bodies are important for
affine isoperimetric inequalities, e.g., \cite{LYZ_Lp},
\cite{LYZ_Sobolev}, \cite{HabSch} and are fundamental in concentration
of volume in convex bodies, e.g,. \cite{KM_unified}, \cite{LW_inf}.
The $L_p$-version of the Busemann-Petty centroid inequality, due to
Lutwak, Yang and Zhang \cite{LYZ_Lp}, concerns the convex body
$Z_p(K)$ defined by its support function
\begin{equation}
  \label{eqn:Zp}
  h^p(Z_p(K), y) = \frac{1}{V_n(K)}\int_{K}\abs{\langle x, y
    \rangle}^p dx
\end{equation} and states that
\begin{equation}
  \label{eqn:LYZ_Lp}
  V_n(Z_p(K)) \gr V_n(Z_p(r_K B)).
\end{equation}
A precursor to \eqref{eqn:LYZ_Lp} is due to Lutwak and Zhang \cite{LZ}
who proved that when $K$ is origin-symmetric,
\begin{equation}
  \label{eqn:LZ_polar}
  V_n(Z_p(K)^{\circ}) \ls V_n(Z_p(r_KB)^{\circ}).
\end{equation}
When $p\rightarrow \infty$, $Z_p(K)$ converges to $Z_{\infty}(K) = K$
and \eqref{eqn:LZ_polar} recovers the classical Blaschke-Santal\'{o}
inequality \cite{Santalo},
\begin{equation}
  \label{eqn:BS}
  V_n(K^{\circ}) \ls V_n((r_KB)^{\circ}).
\end{equation} 
The latter holds more generally for non-symmetric bodies with an
appropriate choice of center.  The analogue of \eqref{eqn:LZ_polar} in
the non-symmetric case was proved by Haberl and Schuster
\cite{HabSch}, to which we refer for further references and background
on $L_p$-Brunn-Minkowski theory.

Inequalities \eqref{eqn:LYZ_Lp} and \eqref{eqn:LZ_polar} are
fundamental inequalities in the $L_p$ Brunn-Minkowski
theory. Recently, such inequalities have been placed in a general
framework involving Orlicz functions by Lutwak, Yang, and Zhang, e.g.,
\cite{LYZ_Orlicz}, \cite{LYZ_Orlicz_projection} and a closely related
concept, due to Gardner, Hug and Weil \cite{GHW_JEMS},
\cite{GHW_Orlicz}, termed $M$-addition, which we discuss in \S
\ref{section:iso}; for further extensions and background, see
\cite{BLYZ}.  We treat stochastic forms of fundamental related
inequalities.  For example, we show that in \eqref{eqn:BM_sd} one can
replace Minkowski addition by $M$-addition.  With the help of laws of
large numbers, this leads to a streamlined approach to all of the
above inequalities and others.

The notion of $M$-addition fits perfectly with the random linear
operator point of view which we have used in our work on this topic
\cite{PaoPiv_probtake}, \cite{PaoPiv_smallball}.  For random vectors
$x_1,\ldots,x_N$, we form the $n\times N$ random matrix
$[x_1,\ldots,x_N]$ and view it as a linear operator from $\R^N$ to
$\R^n$.  If $C\subseteq \R^N$, then
\begin{equation*}
  [x_1,\ldots,x_N] C = \left\{\sum_{i=1}^N c_i x_i: c=(c_i)\in C
  \right\}
\end{equation*}
is a random set in $\R^n$. In particular, if $C = \conv\{e_1,\ldots,
e_N\}$, where $e_1,\ldots,e_N$ is the standard unit vector basis for
$\R^N$, then
\begin{equation*}
  [x_1,\ldots,x_N]\conv\{e_1,\ldots,e_N\} = \conv\{x_1,\ldots,x_N\}.
\end{equation*}Let $B_p^N$ denote the closed unit ball in $\ell_p^N$. 
If $C=B_1^N$, then
\begin{equation*}
  [x_1,\ldots,x_N]B_1^N = \conv\{\pm x_1,\ldots,\pm x_N\}.
\end{equation*}
If $C=B_{\infty}^N$, then one obtains Minkowski sums,
\begin{equation*}
  [x_1,\ldots,x_N] B_{\infty}^N = \sum_{i=1}^N [-x_i,x_i].
\end{equation*}
We define the empirical analogue $Z_{p,N}(K)$ of the $L_p$-centroid
body $Z_p(K)$ by its (random) support function
\begin{equation}
  \label{eqn:Zp_empirical}
  h^p(Z_{p,N}(K), y) = \frac{1}{N} \sum_{i=1}^N \abs{\langle x_i, y
    \rangle}^p,
\end{equation}
where $x_1,\ldots,x_N$ are independent random vectors with density
$\frac{1}{V_n(K)} \mathds{1}_K$; this can be compared with
\eqref{eqn:Zp}; in matrix form, $Z_{p,N}(K) =
N^{-1/p}[x_1,\ldots,x_N]B_q^N$, where $1/p+1/q=1$.  In this framework,
we will explain how uniform measures on Cartesian products of
Euclidean balls arise as extremizers for
\begin{equation}
  \label{eqn:phi_C}
  \mathbb{P}(\phi([X_1,\ldots,X_N]C)> \alpha)
\end{equation} and 
\begin{equation}
  \label{eqn:phi_C_polar}
  \mathbb{P}(\phi(([X_1,\ldots,X_N]C)^{\circ})> \alpha);
\end{equation} 
 over the class of independent random vectors $X_i$ with continuous
 distributions on $\R^n$ having bounded densities; here $C\subseteq
 \R^N$ is a compact convex set (sometimes with some additional
 symmetry assumptions) and $\phi$ an appropriate functional, e.g.,
 volume, surface area, mean width, diameter, among others. Since the
 random sets in the extremizing probabilities are not typically balls
 but sets one generates using balls, there is no clear cut path to
 reduce distributional inequalities for \eqref{eqn:phi_C} and
 \eqref{eqn:phi_C_polar} from one another via duality; for comparison,
 note that the Lutwak-Yang-Zhang inequality for $L_p$ centroid bodies
 \eqref{eqn:LYZ_Lp} implies the Lutwak-Zhang result for their polars
 \eqref{eqn:LZ_polar} by the Blaschke-Santal\'{o} inequality since the
 extremizers in each case are balls (or ellipsoids).

The random operator approach allows one to interpolate between
inequalities for families of convex sets, but such inequalities in
turn yield information about random operators. For example, recall the
classical Bieberbach inequality on the diameter of a convex body
$K\subseteq\R^n$,
\begin{equation}
  \label{eqn:Bieberbach}
  \mathop{\rm diam}(K) \gr \mathop{\rm diam}(r_KB).
\end{equation}A corresponding empirical form is given by 
\begin{equation}
  \label{eqn:Bieberbach_sd}
  \mathbb{P}(\mathop{\rm diam}(K_N) >\alpha)\gr \mathbb{P}(\mathop{\rm
    diam}((r_KB)_N) >\alpha).
\end{equation}
The latter identifies the extremizers of the distribution of certain
operator norms.  Indeed, if $K$ is an origin-symmetric convex body
and we set $K_{N,s} = \conv\{\pm x_1,\ldots,\pm x_N\}$ ($x_i\in \R^n$)
then \eqref{eqn:Bieberbach_sd} still holds and we have the following
for the $\ell_1^N\rightarrow \ell_2^n$ operator norm,
\begin{eqnarray*}
  \mathop{\rm diam}(K_{N,s}) = 2 \norm{[x_1,\ldots,x_N]:\ell_1^N
    \rightarrow \ell_2^n}.
\end{eqnarray*}
We show in \S \ref{section:apps} that if ${\bf X}=[X_1,\ldots,X_N]$,
where the $X_i$'s are independent random vectors in $\R^n$ and have
densities bounded by one, say, then for any $N$-dimensional normed
space $E$, the quantity
\begin{equation*}
  \mathbb{P}\left(\norm{[X_1,\ldots,X_N]:E \rightarrow \ell_2^n} >
  \alpha \right)
\end{equation*}
is minimized when the columns $X_i$ are distributed uniformly in the
Euclidean ball $\wtd{B}$ of volume one, centered at the origin. This
can be viewed as an operator analogue of the Bieberbach inequality
\eqref{eqn:Bieberbach}. When $n=1$, ${\bf X}$ is simply a $1\times N$
row vector and the latter extends to semi-norms. Thus if $F$ is a
subspace of $\R^n$, we get the following for random vectors $x\in
\R^N$ with independent coordinates with densities bounded by one: the
probability
\begin{equation}
  \label{eqn:smallball_norm}
\mathbb{P}(\norm{P_{F}x}_2 >\alpha)
\end{equation} 
is minimized when $x$ is sampled in the unit cube $[-1/2,1/2]^N$ -
products of ``balls'' in one dimension (here $\norm{\cdot}_2$ is the
Euclidean norm and $P_F$ is the orthogonal projection onto
$F$). Combining \eqref{eqn:smallball_norm} with a seminal result by
Ball \cite{Ball_GAFA} on maximal volume sections of the cube, we
obtain a new proof of a result of Rudelson and Vershynin
\cite{RV_IMRN} on small ball probabilities of marginal densities of
product measures (which differs also from the proof in \cite{LPP}, our
joint work G. Livshyts); this is explained in \S \ref{section:apps}.

As mentioned above, Busemann's original motivation for proving the
random simplex inequality \eqref{eqn:Busemann} was to bound suitable
averages of volumes of central hyperplane sections of convex bodies
\eqref{eqn:Busemann_int}.  If $V_n(K)=1$ and $\theta\in S^{n-1}$ then
$V_{n-1}(K\cap \theta^{\perp})$ is the value of the marginal density
of $\mathds{1}_K$ on $[\theta]=\mathop{\rm span}\{\theta\}$ evaluated
at $0$, i.e.
$\pi_{[\theta]}(\mathds{1}_K)(0)=\int_{\theta^{\perp}}\mathds{1}_K(x)dx$.
Thus it is natural that marginal distributions of probability measures
arise in this setting. One reason for placing Busemann-type
inequalities in a probabilistic framework is that they lead to bounds
for marginal distributions of random vectors not necessarily having
independent coordinates, as in our joint work with S. Dann \cite{DPP},
which we discuss further in \S \ref{section:iso}.

Lastly, we comment on some of the tools used to prove such
inequalities.  We make essential use of rearrangement inequalities
such as that of Rogers \cite{Rogers_single}, Brascamp, Lieb and
Luttinger \cite{BLL} and Christ \cite{Christ_kplane}. These interface
particularly well with Steiner symmetrization, shadow systems and
other machinery from convex geometry.  Another key ingredient is an
inequality of Kanter \cite{Kanter} on stochastic dominance.  In fact,
we formulate the Rogers/Brascamp-Lieb-Luttinger inequality in terms of
stochastic dominance using the notion of peaked measures as studied by
Kanter \cite{Kanter} and Barthe \cite{Barthe_unimodal},
\cite{Barthe_central}, among others.  One can actually prove
\eqref{eqn:smallball_norm} directly using the
Rogers/Brascamp-Lieb-Luttinger inequality and Kanter's theorem but we
will show how these ingredients apply in a general framework for a
variety of functionals.  Similar techniques are used in proving
analytic inequalities, e.g., for $k$-plane transform by Christ
\cite{Christ_kplane} and Baernstein and Loss \cite{BaerLoss}. Our
focus is on phenomena in convex geometry and probability.

The paper is organized as follows. We start with definitions and
background in \S \ref{section:prelim}.  In \S \ref{section:dominance},
we discuss the rearrangement inequality of
Rogers/Brascamp-Lieb-Luttinger and interpret it as a result about
stochastic dominance for certain types of functions with a concavity
property, called Steiner concavity, following Christ.  In \S
\ref{section:Steiner}, we present examples of Steiner concave
functions.  In \S \ref{section:iso}, we present general randomized
inequalities. We conclude with applications to operator norms of
random matrices and small deviations in \S \ref{section:apps}.

\section{Preliminaries}
\label{section:prelim}

%\subsection{Convex bodies and Hausdorff distance}
We work in Euclidean space $\R^n$ with the canonical inner-product
$\langle \cdot, \cdot \rangle$ and Euclidean norm $\norm{\cdot}_2$. As
above, the unit Euclidean ball in $\R^n$ is $B=B_2^n$ and its volume
is $\omega_n:=V_n(B_2^n)$; $S^{n-1}$ is the unit sphere, equipped
with the Haar probability measure $\sigma$. Let $G_{n,k}$ be the
Grassmannian manifold of $k$-dimensional linear subspaces of $\R^n$
equipped with the Haar probability measure $\nu_{n,k}$.

A convex body $K\subseteq \R^n$ is a compact, convex set with
non-empty interior.  The set of all compact convex sets in $\R^n$ is
denoted by $\mathcal{K}^n$.  For a convex body $K$ we write $\wtd{K}$
for the homothet of $K$ of volume one; in particular, $\wtd{B} =
\omega_n^{-1/n}B$.  Let $\mathcal{K}^n_{\circ}$ denote the class of
all convex bodies that contain the origin in their interior.  For
$K,L\in \mathcal{K}^n$, the Minkowski sum $K+L$ is the set $\{x+y:x\in
K, y\in L\}$; for $\alpha>0$, $\alpha K = \{\alpha x:x\in K\}$.  We
say that $K$ is origin-symmetric (or simply 'symmetric'), if $-x\in K$
whenever $x\in K$.  For $K\in \mathcal{K}^n$, the support function of
$K$ is given by
\begin{equation*} 
  h_K(x) =\sup\{\langle y, x\rangle\; : \ y\in K\} \quad (x\in \R^n).
\end{equation*}
The mean width of $K$ is
\begin{eqnarray*}
  w(K)= \int_{S^{n-1}}h_K(\theta)+ h_K(-\theta)d\sigma(\theta) 
  = 2\int_{S^{n-1}}h_K(\theta)d\sigma(\theta).
\end{eqnarray*}

Recall that the intrinsic volumes $V_1,\ldots,
V_n$ are functionals on convex bodies which can be defined via the
Steiner formula: for any convex body $K\subseteq \R^n$ and $\eps>0$,
\begin{equation*}
%  \label{eqn:Steiner}
  V_n(K+ \eps B) = \sum_{j=0}^n \omega_{n-j} V_{j}(K){\eps}^{n-j};
\end{equation*} 
here $V_0 \equiv 1$, $V_1$ is a multiple of the
mean width, $2V_{n-1}$ is the surface area and $V_n$ is the volume;
see \cite{Schneider_book_ed2}.

For compact sets $C_1, C_2$ in $\R^n$, we let $\delta^{H}(C_1,C_2)$
denote the Hausdorff distance:
\begin{eqnarray*}
  \delta^{H}(C_1,C_2) &= &\inf\{\eps>0\; : \ C_1\subseteq C_2+\eps B,
  C_2\subseteq C_1+\eps B\}\\
  &= &\sup_{\theta \in S^{n-1}}\abs{h_K(\theta)-h_L(\theta)}.
\end{eqnarray*}  

A set $K\subseteq\R^n$ is star-shaped if it is compact, contains the
origin in its interior and for every $x\in K$ and $\lambda\in[0,1]$ we
have $\lambda x\in K$. We call $K$ a star-body if its radial function
\begin{equation*}
\rho_K(\theta) =\sup\{t>0 : t\theta\in K\} \quad (\theta \in S^{n-1})
\end{equation*}
is positive and continuous. Any positive continuous function
$f:S^{n-1}\rightarrow \R$ determines a star body with radial function
$f$.

%\subsection{Convex measures and $s$-concave functions}

Following Borell \cite{Borell_1}, \cite{Borell_2}, we say that a
non-negative, non-identically zero, function $\psi$ is
$\gamma$-concave if: (i) for $\gamma>0$, $\phi^{\gamma}$ is concave on
$\{ \psi>0\}$, (ii) for $\gamma=0$, $\log{\psi}$ is concave on
$\{\psi>0\}$; (iii) for $\gamma <0$, $\psi^{\gamma}$ is convex on $\{
\psi>0\}$. Let $ s\in [-\infty, 1]$.  A Borel measure $\mu$ on
$\mathbb R^{n}$ is called $s$-concave if
$$ \mu \left ( (1-\lambda) A + \lambda B\right) \gr \left(
 (1-\lambda) \mu(A)^{s} + \lambda \mu(B)^{s} \right)^{\frac{1}{s}} $$
 for all compact sets $A, B\subseteq \mathbb R^{n}$ such that $\mu(A)
 \mu(B) >0$. For $s=0$, one says that $\mu$ is $\log$-concave and the
 inequality reads as
$$ \mu \left ( (1-\lambda) A + \lambda B\right) \gr
 \mu(A)^{1-\lambda} \mu(B)^{\lambda} . $$ Also, for $s=-\infty$, the
 measure is called convex and the inequality is replaced by
$$ \mu \left ( (1-\lambda) A + \lambda B\right) \gr  \min\{ \mu(A),\mu(B)\}. $$

An $s$-concave measure $\mu$ is always supported on some convex subset
of an affine subspace $E$ where it has a density. If $\mu$ is a
measure on $\mathbb R^{n}$ absolutely continuous with respect to
Lebesgue measure with density $\psi$, then it is $s$-concave if and
only if its density $\psi$ is $\gamma$-concave with $\gamma= \frac{s}{
  1-s n } $ (see \cite{Borell_1}, \cite{Borell_2}).

%\subsection{Rearrangements of sets and functions}

Let $A$ be a Borel subset of $\mathbb R^{n}$ with finite Lebesgue
measure. The symmetric rearrangement $A^{\ast}$ of $A$ is the open
ball with center at the origin, whose volume is equal to the measure
of $A$. Since we choose $A^{\ast}$ to be open, ${\bf 1}_{A}^{\ast}$ is
lower semicontinuous. The symmetric decreasing rearrangement of ${\bf
  1}_{A}$ is defined by ${\bf 1}^{\ast}_{A} = {\bf 1}_{A^{\ast}}$. We
consider Borel measurable functions $f:\R^{n}\rightarrow \R_{+}$ which
satisfy the following condition: for every $t>0$, the set $\{ x\in
\R^{n}: f(x) >t\}$ has finite Lebesgue measure. In this case, we say
that $f$ vanishes at infinity. For such $f$, the symmetric decreasing
rearrangement $f^{\ast}$ is defined by
$$f^{\ast} (x) = \int_{0}^{\infty} \mathds{1}_{\{f>t\}^{\ast}} (x) d t =
  \int_{0}^{\infty} \mathds{1}_{\{f>t\}^{\ast}} (x) d t . $$ The latter
    should be compared with the ``layer-cake representation'' of $f$:
\begin{equation}
  \label{eqn:layer_cake}
  f(x)=\int_0^{\infty }1_{\{ f> t\}}(x)dt.
\end{equation}
see \cite[Theorem 1.13]{LL_book}.  Note that the function $f^{\ast}$
is radially-symmetric, radially decreasing and equimeasurable with
$f$, i.e., $\{ f>a\}$ and $ \{ f^{\ast}>a\}$ have the same volume for
each $a>0$. By equimeasurability one has that $ \|f\|_{p} =
\|f^{\ast}\|_{p}$ for each $1\ls p \ls \infty$, where $ \|\cdot
\|_{p}$ denote the $L_{p} (\R^{n})$-norm.
 
Let $f:\R^{n}\rightarrow \R_{+}$ be a measurable function vanishing at
infinity. For $\theta \in S^{n-1}$, we fix a coordinate system that
$e_{1}:= \theta$. The Steiner symmetral $f(\cdot | \theta ) $ of $f$
with respect to $\theta^{\perp}:= \{ y\in \R^{n}: \langle y, \theta
\rangle =0\}$ is defined as follows: for $z:= (x_{2}, \ldots
,x_{n})\in \theta^{\perp}$, we set $f_{z, \theta}(t) = f( t, x_{2},
\ldots , x_{n})$ and define $ f^{\ast}(t, x_{2}, \ldots ,
x_{n}|\theta) := (f_{z,\theta})^{\ast}(t).$ In other words, we obtain
$f^{\ast}(\cdot|\theta)$ by rearranging $f$ along every line parallel
to $\theta$. We will use the following fact, proved in [4]: if
$g:\R^{n}\rightarrow \R_{+}$ is an integrable function with compact
support, there exists a sequence of functions $g_{k}$, where $g_{0}=g$
and $g_{k+1}= g_{k}^{\ast}( \cdot |\theta_{k})$, for some
$\theta_{k}\in S^{n-1}$, such that $\lim_{k\rightarrow \infty} \|g_{k}
- g^{\ast} \|_{1} = 0$.  We refer the reader to the books
\cite{LL_book}, \cite{Simon} or the introductory notes
\cite{Burchard_notes} for further background material on
rearrangement of functions.

\section{Inequalities for stochastic dominance}

\label{section:dominance}

We start with a seminal inequality now known as the
Rogers/Brascamp-Lieb-Luttinger inequality.  It was observed by Madiman
and Wang in \cite{WangMadiman} that Rogers proved the inequality in
\cite{Rogers_single} but it is widely known as the
Brascamp-Lieb-Luttinger inequality \cite{BLL}.  We will state it only
for integrable functions since this is the focus in our paper.

\begin{theorem}
\label{RBLL-th}
 Let $f_{1}, \ldots , f_{M}$ be non-negative integrable functions on
 $\R$ and $u_{1}, \ldots , u_{M} \in \R^N$. Then
\begin{equation}
\label{R-BLL}
\int_{\R^{N}} \prod_{i=1}^{M} f_{i} ( \langle x, u_{i} \rangle
) d x \ls \int_{\R^{N} } \prod_{i=1}^{M} f_{i}^{\ast} (
\langle x, u_{i} \rangle ) dx .
\end{equation} 
\end{theorem}
 
We will write the above inequality in an equivalent form using the
notion of peaked measures. The ideas behind this definition can be
tracked back to Anderson \cite{Anderson} and Kanter \cite{Kanter},
among others, but here we follow the terminology and notation of
Barthe in \cite{Barthe_unimodal}, \cite{Barthe_central}.  Let $
\mu_{1}, \mu_{2}$ be finite Radon measures on $\R^{n}$ with
$\mu_{1}(\R^n) = \mu_{2}(\R^n)$. We say that $\mu_{1}$ is more peaked
than $\mu_{2}$ (and we write $\mu_{1} \succ \mu_{2}$ or $\mu_{2}\prec
\mu_{1}$) if
\begin{equation}
\label{peaked-1}
\mu_{1} ( K) \gr \mu_{2}( K ) 
\end{equation}
for all symmetric convex bodies $K$ in $\R^{n}$. If $X_{1}, X_{2}$ are
random vectors in $\R^{n}$ with distributions $\mu_{1}$ and $\mu_{2}$,
respectively, we write $X_{1} \succ X_{2}$ if $\mu_{1} \succ
\mu_{2}$. Let $f_{1}, f_{2}$ two non-negative integrable functions on
$\R^{n}$ with $\int f_1 = \int f_2$. We write $f_{1} \succ f_{2}$ if
the measures $\mu_{i}$ with densities $f_{i}$ satisfy $ \mu_{1} \succ
\mu_{2} $. It follows immediately from the definition that the
relation $\succ$ is transitive. Moreover if $\mu_{i} \succ \nu_{i}$
and $t_{i}>0$, $1\ls i \ls N$ then $ \sum_{i} t_{i} \mu_{i} \succ
\sum_{i} t_{i} \nu_{i} $. Another consequence of the definition is
that if $ \mu\succ \nu$ and $ E$ is a $k$-dimensional subspace then
the marginal of $\mu$ on $E$, i.e. $\mu\circ P_E^{-1}$, is more peaked
than the marginal of $ \nu$ on $E$. To see this, take any symmetric
convex body $K$ in $E$ and consider the infinite cylinder $C:= K
\times E^{\perp}\subseteq \mathbb R^{n}$. It is enough to check that
$\mu( C) \gr \nu(C)$, and this is satisfied since $C$ can be
approximated from inside by symmetric convex bodies in $\mathbb
R^{n}$. More generally, if $\mu\succ \nu$ then for every linear map
$T$, we have
\begin{equation}
\label{push}
\mu\circ T \succ \nu\circ T,
\end{equation} 
where $\mu\circ T$ is the pushforward measure of $\mu$ through the map $T$. 
  
Recall that $F:\R^{n}\rightarrow \R$ is quasi-concave (quasi-convex)
if for all $s$ the set $ \{ x: F(x) >s\}$ ($ \{ x: F(x) \ls
s\}$) is convex. 

\begin{lemma}
\label{peaked-eq}
Let $\mu_{1}, \mu_{2}$ be Radon measures on $\mathbb R^{n}$ with
$\mu_{1}(\R^n)=\mu_{2}(\R^n)$. Then $\mu_{1} \succ\mu_{2}$ if and only
if
\begin{equation}
\label{succ-equiv}
\int_{\R^{n}} F(x) d\mu_{1}(x) \gr \int_{\R^{n}} F(x) d\mu_{2}(x) 
\end{equation}
 for all even non-negative quasi-concave functions $F$. 
%In fact, the above statement  is an equivalent formulation of the statement $\mu_{1} \succ \mu_{2}$. %In the case that $F$ is quasi-convex, then the inequality is reversed. 
\end{lemma}

\begin{proof}
Assume first that $\mu_{1}\succ \mu_{2}$ and that $F$ is even and
quasi-concave. Then by the layer-cake representation and Fubini's
theorem,
$$ \int_{\R^{n}} F(x) d\mu_{1}(x) = \int_{0}^{\infty} \int_{\{ x: F(x) >s\}} d\mu_{1}(x) d s \gr  $$
$$\int_{0}^{\infty} \int_{\{ x: F(x) >s\}} d\mu_{2}(x) d s=
\int_{\R^{n}} F(x) d\mu_{2}(x) .$$ Conversely, if $K$ is a symmetric
convex body then $F:= {\bf 1}_{K}$ is even and quasi-concave and
\eqref{succ-equiv} becomes $ \mu_{1} ( K) \gr \mu_{2}(K)$ so
\eqref{succ-equiv} implies that $\mu_{1} \succ \mu_{2}$.  \end{proof}

\smallskip

 We are now able to state the following equivalent
form of the Rogers/Brascamp-Lieb-Luttinger inequality:

\begin{proposition}
\label{prop-RBLL}
 Let $f_{1}, \ldots , f_{N}$ be non-negative integrable functions on
 $\R$. Then
\begin{equation}
\label{R-BLL-new}
\prod_{i=1}^{N} f_{i} \prec \prod_{i=1}^{N} f_{i}^{\ast} .
\end{equation}
\end{proposition}

\smallskip

 Let us explain why Theorem \ref{RBLL-th} implies Proposition
 \ref{prop-RBLL}. Note first that without loss of generality we can
 replace the assumption ``integrable" with ``having integral $1$."
 Let $K$ be a symmetric convex body in $\R^{N}$. Then it can be
 approximated by intersections of symmetric slabs of the form
$$ K_{m} := \bigcap_{i=1}^{m} \{ x\in \R^{N} : | \langle x , u_{i}
 \rangle | \ls 1 \} $$ for suitable $u_{1}, \ldots , u_{m}\in
 \R^{N}$. Note that $ {\bf 1}_{K_{m}} = \prod_{i=1}^{m} {\bf
   1}_{[-1,1]} ( \langle \cdot , u_{i} \rangle ) $. Apply
 \eqref{R-BLL} with $M= m+N$ and $u_{m+i}:= e_{i}$,
 $i=1,\ldots,N$. Then (since $ {\bf 1}_{K_{m}} \rightarrow {\bf
   1}_{K}$ in $L_{1}$), we get that
 \begin{equation}
   \label{eqn:intK}
    \int_{K} \prod_{i=1}^{N} f_{i}(x_{i}) d x \ls \int_{K}
 \prod_{i=1}^{N} f_{i}^{\ast}(x_{i}) d x . 
 \end{equation}
 Since $K$ is an arbitrary symmetric convex body in $\mathbb R^{N}$,
 we get \eqref{R-BLL-new}.  The latter is an extension of a theorem of
 Anderson \cite{Anderson} and it is the basis of Christ's extension of
 the Rogers/Brascamp-Lieb-Luttinger inequality \cite{Christ_kplane};
 see also the thesis of Pfiefer \cite{Pfiefer} and work of Baernstein
 and Loss \cite{BaerLoss}.

In the other direction, consider non-negative integrable functions
$f_{1}, \ldots , f_{m}$ and let $u_{1}, \ldots , u_{m}$ be vectors in
$\R^{N}$. Write $ F(x):= \prod_{i=1}^{m} f_{i}(x_{i})$ and $ F_{\ast}
(x) := \prod_{i=1}^{m} f_{i}^{\ast}(x_{i})$. Let $ T$ be the $m\times
N$ matrix with rows $u_1,\ldots,u_m$. Note that \eqref{R-BLL-new}
implies that $ F\prec F_{\ast}$. By \eqref{push} we also have that $
F\circ T \prec F_{\ast} \circ T$ so that for any symmetric convex body
$K\subseteq \R^N$, $\int_K F\circ T(x) dx \leq \int_K F_{\ast}\circ T(x)dx$,
hence
$$ \int_{\R^{N} } \prod_{i=1}^{m} f_{i}( \langle x, u_{i}\rangle ) d x
\ls \int_{\R^{N} } \prod_{i=1}^{m} f_{i}^{\ast}( \langle x,
u_{i}\rangle ) d x $$ which is \eqref{R-BLL}.

\smallskip

 Actually we will use the Rogers/Brascamp-Lieb-Luttinger
inequality in the following form \cite{Christ_kplane}.

\begin{corollary}
\label{cor-BLL}
 Let $f_{1}, \ldots , f_{m}$ be non-negative integrable functions on
 $\R$.  Let $u_{1}, \ldots , u_{m}$ be non-zero vectors in $\R^{N}$
 and let $ F_{1}, \ldots , F_{M}$ be non-negative, even, quasi-concave
 functions on $\R^{N}$. Then
\begin{equation}
\label{cor-BLL-1}
\int_{\R^{N}}  \prod_{j=1}^{M}  F_{j} ( x) \prod_{i=1}^{m} f_{i} ( \langle x, u_{i} \rangle ) d x \ls \int_{\R^{N}}  \prod_{j=1}^{M}  F_{j} ( x) \prod_{i=1}^{m} f_{i}^{\ast} ( \langle x, u_{i} \rangle ) d x .
\end{equation}
 Also, if $F$ is a non-negative, even, quasi-convex function on
 $\R^{N}$, we have 
\begin{equation}
\label{cor-BLL-2}
\int_{\R^{N}} F(x) \prod_{i=1}^{N} f_{i} (x_{i}) d x \gr
\int_{\R^{N}} F(x) \prod_{i=1}^{m} f_{i}^{\ast} ( x_{i}) d x .
\end{equation}
\end{corollary}  

\begin{proof}
 {\it (Sketch)}. Note that $ \prod_{j=1}^{M} F_{j}(x)$ is again
 quasi-concave and even.  So \eqref{cor-BLL-1} follows from
 Proposition \ref{prop-RBLL} and Lemma \ref{peaked-eq}.

For the proof of \eqref{cor-BLL-2} first notice that it is enough to
prove in the case that $\int_{\R} f_{i}(t) d t = 1$, $1\ls i \ls
N$. Recall that for every $t>0$, $\{ F\ls t\}$ is convex and
symmetric. Thus using Proposition \ref{prop-RBLL} and Lemma
\ref{peaked-eq}, we get 
\begin{eqnarray*}
\lefteqn{\int_{\R^{N}} F(x) \prod_{i=1}^{N}f_{i}(x_{i}) d x}\\ & & =
\int_{\R^{N}} \left( \int_{0}^{\infty} {\bf 1}_{\{F>t\}} (x) d t
\right) \prod_{i=1}^{N} f_{i} ( x_{i}) d x \\ & & = \int_{0}^{\infty}
\int_{\R^{N}} (1- {\bf 1}_{\{F\ls t \} }) \prod_{i=1}^{N} f_{i} (
x_{i}) d x d t \\ & & = \int_{0}^{\infty} \left( \int_{\R^{N}}
\prod_{i=1}^{N} f_{i}^{\ast} (x_{i}) d x - \int_{\R^{N} } {\bf
  1}_{\{F\ls t \} }) \prod_{i=1}^{N} f_{i} ( x_{i}) d x \right) d t
\\ & & \gr \int_{0}^{\infty} \left( \int_{\R^{N}} \prod_{i=1}^{N}
f_{i}^{\ast} (x_{i}) d x - \int_{\R^{N} } {\bf 1}_{\{F\ls t \} })
\prod_{i=1}^{N} f_{i}^{\ast} ( x_{i}) d x \right) d t \\& &
=\int_{\R^{N}} F(x) \prod_{i=1}^{N} f_{i}^{\ast}(x_{i}) d x.
\end{eqnarray*}
\end{proof}

%\noindent {\it Remark.} From the proof is it also clear that in the
%formulation of the above corollary, we may replace the integration
%with respect with the Lebesgue measure with another measure that has
%an even and quasi-concave density.

\smallskip

%%%%%%%%%%%%%%%%%%%%%%%%%%%%%%%%%%%%%%%%%%%%%%%%%%%%%%%%%%%%%%%%%%%%%%%%%%%%%%%%%%%%%%%%%%%%%%%%%%%%%%%%%%%%%%%%%%%%%%%%%%%%%%%%%%%%%%%%%%%%%%%%%%%%

We say that a function $f$ on $\mathbb R^{n}$ is unimodal if it is the
increasing limit of a sequence of functions of the form,
$$ \sum_{i=1}^{m} t_{i} {\bf 1}_{K_{i}}, $$ where $t_{i} \gr 0$ and $
K_{i}$ are symmetric convex bodies in $\mathbb R^{n}$. Even
quasi-concave functions are unimodal and every even and non-increasing
function on $\R_{+}$ is unimodal.  In particular, for every integrable
$f:\mathbb R^{n}\rightarrow \R_{+}$, $f^{\ast}$ is unimodal. We will
use the following lemma, which is essentially the bathtub principle
(e.g., \cite{LL_book}).
\pagebreak
\begin{lemma}
\label{unim-1}
 Let $ f:\R^{n}\rightarrow \R_{+}$ be an integrable function.
\begin{enumerate}
\item If $g:\R_{+}$ is a measurable function, $\beta:=
  \int_{0}^{\infty} g(t) t^{n-1} d t <\infty$ and $ \phi:\mathbb
  R_{+}\rightarrow \mathbb R_{+}$ is a non-decreasing function, then
\begin{equation}
\label{Lemma35}
\int_{0}^{\infty}\phi(t) g(t) t^{n-1} dt\gr \int_{0}^{\infty} \phi(t)
h(t) t^{n-1} d t ,
\end{equation} 
where $h:= {\bf 1}_{[0, (n\beta)^{\frac{1}{n}}]}$. If $\phi$ is
non-increasing, then the inequality in \eqref{Lemma35} is reversed.
\item If $n=1$, $\|f\|_{1}=1$, $\|f\|_{\infty} \ls 1$ and $f$ is even,
  then $ f^{\ast} \prec {\bf 1}_{[-\frac{1}{2}, \frac{1}{2}]}$.
\item If $f$ is rotationally invariant, $\|f\|_{1}=1$, and
  $\|f\|_{\infty} \ls 1$, then for every star-shaped set $K\subseteq
  \R^{n}$, $ \int_{K} f(x) d x \ls \int_{K} {\bf 1}_{\wtd{B}} (x) d x
  $.
\item If $\|f\|_{1}=1$, $\|f\|_{\infty} \ls 1$, then $ f^{\ast} \prec
  {\bf 1}_{\wtd{B}} $.
\end{enumerate}
\end{lemma}

\begin{proof} 
  The proof of the first claim is standard, see e.g. \cite[Lemma
    3.5]{PaoPiv_probtake}. The second claim follows from the first, by
  choosing $n=1$, $\beta= \frac{1}{2}$ and $\phi:= {\bf 1}_{[0,a]}$,
  $a>0$.  The third claim follows by applying \eqref{Lemma35} after
  writing the desired inequality in polar coordinates. The last claim
  follows immediately from the third.
\end{proof}

A fundamental result on peaked measures is the following result of
Kanter \cite{Kanter}

\begin{theorem}
  \label{Kanter}
  Let $f_{1}, f_{2}$ be functions on $\R^{n_{1}}$ such that $ f_{1}
  \succ f_{2}$ and $f$ a unimodal function on $\R^{n_{2}}$. Then
  \begin{equation}
     f f_{1} \succ f f_{2} .
    \end{equation}
In particular, if $f_{i} , g_{i}$ are unimodal functions on
$\R^{n_{i}}$, $1\ls i \ls M$ and $f_{i} \succ g_{i}$ for all $i$, then
  \begin{equation} \prod_{i=1}^{M} f_{i} \succ  \prod_{i}^{M} g_{i} . \end{equation}
\end{theorem}

\begin{proof}(Sketch) 
Without loss of generality, assume $\int f_1 = \int f_2 =\int f =
1$. Consider first the case where $f:= {\bf 1}_{L}$ for some symmetric
convex body $L$ in $\R^{n_{2}}$. Let $K$ be a symmetric convex body in
$\R^{n_{1}} \times \R^{n_{2}}$. The Pr\'ekopa-Leindler inequality
implies that the even function
$$ F( x) := \int_{\R^{n_{2}}} {\bf 1}_{K} ( x,y) {\bf 1}_{L} ( y) d
  y $$ is $\log$-concave. So, using Lemma \ref{peaked-eq}, 
$$ \int_{\R^{n_{1}} }\int_{ \R^{n_{2}} }{\bf 1}_{K}( x,y) f_{1}(x) f(y) d x d y = \int_{\R^{n_{1}} }F(x) f_{1} (x) d x \gr $$
$$ \int_{\R^{n_{1}}} F(x) f_{2} (x) d x = \int_{\R^{n_{1}} }\int_{
    \R^{n_{2}} }{\bf 1}_{K}( x,y) f_{2}(x) f(y) d x d y,$$ hence $f f_{1} \succ f f_{2}$. The
  general case follows easily.  \end{proof}

\smallskip

Theorem \ref{Kanter} and Lemma \ref{unim-1} immediately imply the
following corollary.

\begin{corollary}
\label{Kanter-cor}
 Let $ f_{1}, \ldots , f_{m}:\R^{n}\rightarrow \R_{+}$ be probability
 densities of continuous distributions such that $\max_{i\ls M}
 \|f_{i} \|_{\infty} \ls 1$.  If $n=1$, then
\begin{equation}
\label{Kanter-cor-1}
\prod_{i=1}^{m} f_{i}^{\ast} \prec {\bf 1}_{Q_{m}}
\end{equation} 
where $Q_{m}$ is the $m$-dimensional cube of volume $1$ centered at
0. In the general case we have that
\begin{equation}
\label{Kanter-cor-2}
\prod_{i=1}^{m} f_{i}^{\ast} \prec\prod_{i=1}^{m} {\bf 1}_{\wtd{B}} .
\end{equation} 
\end{corollary}

\bigskip

%%%%%%%%%%%%%%%%%%%%%%%%%%%%%%%%%%%%%%%%%%%%%%%%%%%%%%%%%%%%%%%%%%%%%%%%%%%%%%%%%%%%%%%%%%%%%%%%%%%%%%%%%%%%%%%%%%%%%%%%%%%%%%%%%%%%%%%%%%%%%%%%%%%%%%%%%%%%%%%%%%%%%%%%%%%%%%%%%%%%%%%%%%%%%%%%%%%%
%%%%%%%%%%%%%%%%%%%%%%%%%%%%%%%%%%%%%%%%%%%%%%%%%%%%%%%%%%%%%%%%%%%%%%%%%%%%%%%%%%%%%%%%%%%%%%%%%%%%%%%%%%%%%%%%%%%%%%%%%%%%%%%%%%%%%%%%%%%%%%%%%%%%

\subsection{Multidimensional case}

\medskip
 Let $f$ be a non-negative function on $\R^{n}$, $\theta \in S^{n-1}$
 and $z\in \theta^{\perp}$. We write $f_{z, \theta} (t) :=
 f_{z}(\theta):= f( z+ t \theta)$.  Let $G$ be a non-negative function
 on the $N$-fold product $\mathbb R^{n} \times \ldots \times \mathbb
 R^{n}$. Let $\theta \in S^{n-1}$ and let $Y:= \{ y_{1}, \ldots ,
 y_{N}\} \subseteq \theta^{\perp}:= \{ y\in \mathbb R^{n}: \langle y,
 \theta \rangle =0\}$. We define a function $G_{Y} : \mathbb R^{N}
 \rightarrow \mathbb R_{+}$ as
$$ G_{Y, \theta} ( t_{1}, \ldots , t_{N}) := G ( y_{1} + t_{1} \theta
 , \ldots , y_{N} + t_{N} \theta ) . $$ We say that $G:\mathbb R^{n}
 \times \ldots \times \mathbb R^{n}\rightarrow \R_{+}$ is Steiner
 concave if for every $\theta$ and $Y\subseteq \theta^{\perp}$ we have
 that $G_{Y, \theta}$ is even and quasi-concave; similarly, we say $G$
 is Steiner convex if $G_{Y,\theta}$ is even and quasi-convex. For
 example, if $N=n$, then negative powers of the absolute value of the
 determinant of an $n\times n$ matrix are Steiner concave since the
 determinant is a multi-linear function of its columns (or rows). Our
 results depend on the following generalization of the Rogers and
 Brascamp-Lieb-Luttinger inequality due to Christ \cite{Christ_kplane}
 (our terminology and presentation is suited for our needs and differs
 slightly from \cite{Christ_kplane}).

\begin{theorem}
\label{MultiBLL}
 Let $f_{1},\ldots , f_{N}$ be non-negative integrable functions on
 $\mathbb R^{n}$, $A$ an $ N \times \ell$ matrix. Let $ F^{(k)} :
 (\mathbb R^{n})^{\ell}\rightarrow \R_{+}$
 be Steiner concave functions $1\ls k \ls M$ and let $\mu$ be a
 measure with a rotationally invariant quasi-concave density on
 $\mathbb R^{n}$. Then
$$ \int_{\mathbb R^{n}} \ldots \int_{\mathbb R^{n}} \prod_{k=1}^{M}
 F^{(k)} ( x_{1}, \ldots , x_{\ell}) \prod_{i=1}^{N} f_{i} \left(
 \sum_{j=1}^{\ell} a_{ij} x_{j}\right) d \mu(x_{\ell}) \ldots
 d\mu(x_{1} )\ls $$
\begin{equation}
\label{Multi-BLL-1}
 \int_{\mathbb R^{n}} \ldots \int_{\mathbb R^{n}} \prod_{k=1}^{M}
 F^{(k)} ( x_{1}, \ldots , x_{\ell}) \prod_{i=1}^{N} f_{i}^{\ast}
 \left( \sum_{j=1}^{\ell} a_{ij} x_{j}\right) d \mu(x_{\ell}) \ldots
 d\mu(x_{1}) .
\end{equation}
\end{theorem}

\begin{proof}(Sketch)
  Note that in the case $n=1$, \eqref{Multi-BLL-1} is just
  \eqref{cor-BLL-1}. We consider the case $n>1$.  Let $u_{i}\in
  \mathbb R^{\ell}$ be the rows of the matrix $A$.  Fix a direction
  $\theta\in S^{n-1}$ and let $ y_{1}, \ldots , y_{\ell}\in
  \theta^{\perp}$ the (unique) vectors such that $ x_{j} = y_{j} +
  t_{j} \theta$. Consider the function
$$ h_{i} ( \langle u_{i}, t \rangle ) := f_{i} \left(
  \sum_{j=1}^{\ell} a_{ij} (y_{j}+ t_{j} \theta)\right) , \ 1\ls i \ls
  N . $$ We defined the Steiner symmetral $f_{i}^{\ast} (\cdot |
  \theta ) = h_i^{\ast}$ in the direction $\theta$ in \S
  \ref{section:prelim}.  Then by Fubini's theorem we write each
  integral as an integral on $\theta^{\perp}$ and
  $[\theta]=\mathop{\rm span}\{\theta\}$, for each fixed $y_{1},
  \ldots , y_{\ell}$ we apply \eqref{cor-BLL-1} for the functions
  $h_{i}$ and the quasi-concave functions $F^{(k)}_{Y,
    \theta}$. (Recall the definition of Steiner concavity). Using
  Fubini's theorem again, we have proved that
$$ \int_{\mathbb R^{n}} \ldots \int_{\mathbb R^{n}} \prod_{k=1}^{M}
 F^{(k)} ( x_{1}, \ldots , x_{\ell}) \prod_{i=1}^{N} f_{i} \left(
 \sum_{j=1}^{\ell} a_{ij} x_{j}\right) d \mu(x_{\ell}) \ldots
 d\mu(x_{1} )\ls $$
\begin{equation}
\label{Multi-BLL-11}
 \int_{\mathbb R^{n}} \ldots \int_{\mathbb R^{n}} \prod_{k=1}^{M}
 F^{(k)} ( x_{1}, \ldots , x_{\ell}) \prod_{i=1}^{N} f_{i}^{\ast}
 \left( \sum_{j=1}^{\ell} a_{ij} x_{j} | \theta \right) d
 \mu(x_{\ell}) \ldots d\mu(x_{1}) .
\end{equation}
In \cite{BLL} it has been proved that the function $f^{\ast}$ can be
approximated (in the $L_{1}$ metric) by a suitable sequence of Steiner
symmetrizations. This leads to \eqref{Multi-BLL-1}.
\end{proof}

Let $F$ be a Steiner concave function. Notice that the function
$\tilde{F} := {\bf 1}_{\{ F>\alpha\}} $ is also Steiner
concave. Indeed, if $\theta \in S^{n-1}$ and $Y\subseteq
\theta^{\perp}$, notice that $\tilde{F}_{Y, \theta} (t) = 1 $ if and
only if $ F_{Y, \theta}(t) >\alpha$. Since $F$ is Steiner concave,
$\tilde{F}_{Y, \theta}$ is the indicator function of a symmetric
convex set. So $\tilde{F}$ is also Steiner concave. Thus we have the
following corollary.

\begin{corollary}
\label{cor-multi-BLL}
Let $F: \mathbb R^{n} \times \ldots \times \mathbb R^{n}\rightarrow \R_{+}$ be a Steiner
concave function and let $f_{i}:\mathbb R^{n}\rightarrow \mathbb
R_{+}$ be non-negative functions with $\norm{f_i}_1=1$ for $1\ls i \ls
N$. Let $\nu$ be the (product) probability measure defined on $\mathbb
R^{n}\times \ldots \times \mathbb R^{n}$ with density $\prod_i f_{i}$
and let $\nu^{\ast}$ have density $\prod_i f_{i}^{\ast}$.  Then for
each $\alpha>0$,
\begin{equation}
\label{cor-Multi-BLL-1}
\nu \left( \{ F (x_{1}, \ldots , x_{N}) > \alpha \} \right) \ls
\nu^{\ast} \left( \{ F (x_{1}, \ldots , x_{N} ) > \alpha \} \right) .
\end{equation}
 Moreover, if $G: \mathbb R^{n} \times \ldots \times \mathbb R^{n}\rightarrow \R_{+}$ is
 a Steiner convex function, then
\begin{equation}
\label{cor-Multi-BLL-2}
\nu \left( \{ G (x_{1}, \ldots , x_{N}) > \alpha \} \right) \gr
\nu^{\ast} \left( \{ G (x_{1}, \ldots , x_{N} ) > \alpha \} \right).
\end{equation}
\end{corollary}

\begin{proof} 
We apply \eqref{Multi-BLL-1} for $\mu$ the Lebesgue measure, $\ell=N$,
$A$ the identity matrix, $M=1$ and for the function $\tilde{F}$ (as
defined above). This proves \eqref{cor-Multi-BLL-1}. Working with the
function $1-\tilde{F}$ as in the proof of \eqref{cor-BLL-2} we get
\eqref{cor-Multi-BLL-2}.
\end{proof}

%%%%%%%%%%%%%%%%%%%%%%%%%%%%%%%%%%%%%%%%%%%%%%%%%%%%%%%%%%%%%%%%%%%%%%%%%%%%%%%%%%%%%%%%%%%%%%%%%%%%%%%%%%%%%%%%%%%%%%%%%%%%%%%%%%%%%%%%%%%%%%%%%%%%%%%%%%%%%%%%%%%%%%%%%%%%%%%%%%%%%%%%%%%%%%%%%%%%%%%%%%%%%%%%%%%%%%%%%%%%%%%%%%%%%%%%%%%%%%%%%%%%%%%%%%%%%%%%%%%%%%%%%%%%%%%%%%%%%%%%%%%%%%%%%%%%%%%%%%%%%%%%%%%%%%%%%%%%%%%%%%%%%%%%%%%%%%%%%%%%%%%%%%%%%%%%%%

\subsection{Cartesian products of balls as extremizers}

In the last section, we discussed how in the presence of Steiner
concavity, one can replace densities by their symmetric decreasing
rearrangements. Among products of bounded, radial, decreasing
densities, the uniform measure on Cartesian products of balls arises
in extremal inequalities under several conditions and we discuss two
of them in this section.

We will say that a function $F:\R^n \times \ldots \times \R^n
\rightarrow \R_{+}$ is coordinate-wise decreasing if for any $x_{1},
\ldots , x_{N}\in \mathbb R^{n}$, and $ 0\ls s_{i}\ls t_{i}, 1\ls i
\ls N$,
\begin{eqnarray}
\label{coord-wise-decr}
F( s_{1} x_{1}, \ldots , s_{N} x_{N} ) \gr F ( t_{1} x_{1}, \ldots ,
t_{N} x_{N}) .
\end{eqnarray}

The next proposition can be proved by using Fubini's theorem
iteratively and Lemma \ref{unim-1} (as in \cite{CEFPP}).

\begin{proposition}
\label{pass-ball-1}
 Let $ F : (\mathbb R^{n})^N\rightarrow \mathbb R_{+}$ be a function
 that is coordinate-wise decreasing.  If $g_{1}, \ldots , g_{N}:
 \mathbb R^{n}\rightarrow \mathbb R_{+}$ are rotationally invariant
 densities with $ \max_{i\ls N} \|g_{i}\|_{\infty} \ls 1$, then
\begin{eqnarray}
\label{passball}
  \lefteqn{\int_{\mathbb R^{n}} \ldots \int_{\mathbb R^{n}} F( x_{1},
    \ldots , x_{N} ) \prod_{i=1}^{N} g_{i} (x_{i}) d x_{N} \ldots
    dx_{1}} \\& & \ls
 \int_{\mathbb R^{n}} \ldots \int_{\mathbb R^{n}} F( x_{1}, \ldots ,
 x_{N} ) \prod_{i=1}^{N} {\bf 1}_{\wtd{B}} (x_{i}) d x_{N} \ldots
 dx_{1}.
\end{eqnarray}
\end{proposition}

Using Corollary \ref{Kanter-cor}, we get the following.

\begin{proposition}
  \label{prop:passball_Kanter}
  Let $ F : (\R^n)^N\rightarrow \mathbb R_{+}$ be quasi-concave and
  even.  If $g_{1}, \ldots , g_{N}:
 \mathbb R^{n}\rightarrow \mathbb R_{+}$ are rotationally invariant
 densities with $ \max_{i\ls N} \|g_{i}\|_{\infty} \ls 1$, then
  \begin{eqnarray}
    \label{passball_Kanter}
    \lefteqn{\int_{\mathbb R^{n}} \cdots \int_{\mathbb R^{n}} F( x_{1}, \ldots ,
    x_{N} ) \prod_{i=1}^{N} g_{i} (x_{i}) d x_{N} \ldots dx_{1}} \\ & & \ls 
    \int_{\mathbb R^{n}} \ldots \int_{\mathbb R^{n}} F( x_{1}, \ldots ,
    x_{N} ) \prod_{i=1}^{N} {\bf 1}_{\wtd{B}} (x_{i}) d x_{N} \ldots
    dx_{1}.  
  \end{eqnarray}
\end{proposition}

\section{Examples of Steiner concave and convex functions}
\label{section:Steiner}

As discussed in the previous section, the presence of Steiner
concavity (or convexity) allows one to prove extremal inequalities
when the extremizers are rotationally invariant.  The requisite
Steiner concavity is present for many functionals associated with
random structures.  As we will see, in many important cases, verifying
the Steiner concavity condition is not a routine matter but rather
depends on fundamental inequalities in convex geometry.  In this
section we give several non-trivial examples of Steiner concave (or
Steiner convex) functions and we describe the variety of tools that
are involved.

\subsection{Shadow systems and mixed volumes}

Shadow systems were defined by Shephard \cite{Shephard_shadow} and
developed by Rogers and Shephard \cite{RS}, and Campi and Gronchi,
among others; see, e.g., \cite{CCG_Sylvester}, \cite{CG},
\cite{CG_Lp}, \cite{CG_volume_product}, \cite{Saroglou} and the
references therein. Let $C$ be a closed convex set in $\R^{n+1}$. Let
$(e_1, \ldots, e_{n+1})$ be an orthonormal basis of $\R^{n+1}$ and
write $\R^{n+1}=\R^n\oplus\R e_{n+1}$ so that $\R^n=e_{n+1}^\bot$
. Let $\theta\in S^{n-1}$. For $t\in\R$ let $P_t$ be the projection
onto $\R^n$ parallel to $e_{n+1}-t\theta$: for $x\in \R^n$ and $s\in
\R$,
$$P_t(x+se_{n+1})= x + t s \theta.$$ Set $K_t=P_t C\subseteq \R^n$.
Then the family $(K_t)$ is a shadow system of convex sets, where $t$
varies in an interval on the real line. Shephard \cite{Shephard}
proved that for each $1\ls j\ls n$,
\begin{equation*}
  [0,1]\ni t \mapsto V_j(P_t C)
\end{equation*} 
is a convex function; see work of Campi and Gronchi, e.g.,
\cite{CG_survey}, \cite{CG_Lp} for further background and
references. Here we consider the following $N$-parameter variation,
which can be reduced to the one-parameter case.

\begin{proposition}
  \label{prop:shadow_one}
  Let $n, N$ be postive integers and $C$ be a compact convex set in
  $\R^{n}\times \R^N$. Let $\theta\in S^{n-1}\subseteq \R^n$. For
  $t\in\R^N$ and $(x,y)\in\R^n\times\R^N$, we define
  $P_t(x,y)=x+\langle y,t\rangle\theta$.  Then for all $1\ls j\ls n$,
  \begin{equation*}
    \R^N \ni t\mapsto V_j(P_t C)
  \end{equation*}is a convex function.
\end{proposition}

\begin{proof} (Sketch) Fix $s$ and $t$ in $\R^N$.  It is sufficient to show that 
  \begin{equation*}
    [0,1]\ni \lambda \mapsto V_j(P_{s+\lambda (t-s)} C)
  \end{equation*} 
  is convex. Note that $\lambda \mapsto P_{s+\lambda(s-t)}C$ is a
  one-parameter shadow system and we can apply Shephard's result
  above; for an alternate argument, following Groemer \cite{Groemer},
  see \cite{PaoPiv_probtake}.
\end{proof}

\begin{corollary}
  \label{cor:shadow_one}
  Let $C$ be a compact convex set in $\R^N$. Then for all $1\ls j\ls n$,
  \begin{equation*}
    (\R^{n})^N \ni (x_1,\ldots,x_N)\mapsto V_j([x_1,\ldots,x_N]C)
  \end{equation*}
  is Steiner convex on $\R^N$. Moreover, if $C$ is $1$-unconditional
  then the latter function is coordinate-wise increasing analogous to
  definition \eqref{coord-wise-decr}.
\end{corollary}

\begin{proof}
  Let $\theta \in S^{n-1}$ and $y_{i}\in \theta^{\perp}$ for
  $i=1,\ldots,N$. Write $x_{i} = y_{i}+t_i\theta$. Let
  $\mathcal{C}=[y_{1}+e_{n+1},\ldots,y_{N}+e_{n+N}]C$. Then
  $\mathcal{C}$ is a compact convex set in $\R^n\times\R^N$ which is
  symmetric with respect to $\theta^\bot$ in $\R^{n+N}$ since
  $[y_{1}+e_{n+1},\ldots,y_{N}+e_{n+N}]C\subseteq \theta^{\perp}$. Let
  $P_t : \R^n\times\R^N\to\R^n$ be defined as in Proposition
  \ref{prop:shadow_one}. Then
  \begin{eqnarray*}
    P_t([y_{1}+e_{n+1},\ldots,y_{N}+e_{n+N}]C= [y_{1}+t_1\theta,
      \dots, y_{N}+t_N\theta]C.
  \end{eqnarray*}
  We apply the previous proposition to obtain the convexity claim. Now
  for each $\theta\in S^{n-1}$ and $y_1,\ldots,y_N\in \theta^{\perp}$,
  the sets $[y_1 +t_1\theta,\ldots,y_N+t_N\theta]C$ and $[y_1
    -t_1\theta,\ldots,y_N-t_N\theta]C$ are reflections of one another
  and so the evenness condition (for Steiner convexity) holds as well.
  The coordinate-wise monotonicity holds since one has the following
  inclusion when $C$ is $1$-unconditional: for $0\ls s_i \ls t_i$,
  \begin{equation*}
    [s_1x_1,\ldots,s_N x_N]C \subseteq [t_1 x_1,\ldots,t_N x_N]C.
  \end{equation*}
\end{proof}

\subsection{Dual setting}

Here we discuss the following dual setting involving the polar dual of
a shadow system.  Rather than looking at projections of a fixed
higher-dimensional convex set as in the previous section, this
involves intersections with subspaces. We will invoke a fundamental
inequality concerning sections of symmetric convex sets, known as
Busemann's inequality \cite{Busemann_slice}. This leads to a
randomized version of an extension of the Blaschke-Santal\'{o}
inequality to the class of convex measures (defined in \S
\ref{section:prelim}). For this reason we will need the following
extension of Busemann's inequality to convex measures from our joint
work with D. Cordero-Erausquin and M. Fradelizi \cite{CEFPP}; this
builds on work by Ball \cite{Ball_log}, Bobkov \cite{Bobkov}, Kim,
Yaskin and Zvavitch \cite{KYZ}).

\begin{theorem}
(Busemann Theorem for convex measures). Let $\nu$ be a convex measure
  with even density $\psi\in \mathbb R^{n}$. Then the function $\Phi$
  defined on $\mathbb R^{n}$ by $\Phi(0)=0$ and for $z\neq 0$,
$$ \Phi (z) = \frac{ \|z\|_{2} }{  \int_{z^{\perp}} \psi(x) d x } $$
is a norm.  
\end{theorem} 

The latter inequality is the key to the following theorem from
\cite{CEFPP} which extends the result of Campi-Gronchi
\cite{CG_volume_product} to the setting of convex measures; the
approach taken in \cite{CG_volume_product} was the starting point for
our work in this direction.

\begin{proposition}
  \label{prop:shadow_2}
  Let $\nu$ be a measure on $\mathbb R^{n}$ with a density $\psi$
  which is even and $\gamma$-concave on $\mathbb R^{n}$ for some
  $\gamma\gr -\frac{1}{ n+1}$. Let $ (K_{t}):= P_{t} C$ be an
  $N$-parameter shadow system of origin symmetric convex sets with
  respect to an origin symmetric body $C\subseteq \mathbb R^{n}\times
  \mathbb R^{N}$. Then the function $\R^N\ni t\mapsto
  \nu(K_{t}^{\circ})^{-1}$ is convex.
\end{proposition}

This result and the assumption on the symmetries of $C$ and $\nu$
leads to the following corollary. The proof is similar to that given
in \cite{CEFPP}.

\begin{corollary}
  \label{cor:shadow_two}
  Let $r\gr 0$, $C$ be an origin-symmetric convex set in $\R^N$.  Let
  $\nu$ be a radial measure on $\R^n$ with a density $\psi$ which is
  $-1/(n+1)$-concave on $\R^n$. Then the function $$ G( x_{1}, \ldots
  , x_{N}) = \nu ( ( [x_{1}\ldots x_{N}] C + r B_{2}^{N} )^{\circ}
  ) $$ is Steiner concave. Moreover if $C$ is $1$-unconditional then
  the function $G$ is coordinate-wise decreasing.
\end{corollary}

\begin{remark} 
  The present setting is limited to origin-symmetric convex
  bodies. The argument of Campi and Gronchi \cite{CG_volume_product}
  leading to the Blaschke-Santal\'{o} inequality has been extended to
  the non-symmetric case by Meyer and Reisner in
  \cite{MeyerReisner}. It would be interesting to see an asymmetric
  version for random sets as it would give an empirical form of the
  Blaschke-Santal\'{o} inequality and related inequalities,
  e.g,. \cite{HabSch} in the asymmetric case.
\end{remark}

\subsection{Minkowski addition and extensions}
\label{sub:Madd}

In this section, we recall several variations of Minkowski addition
that are the basis of $L_p$-Brunn-Minkowski theory, $p\gr 1$, and its
extensions. $L_p$-addition as originally defined by Firey \cite{Firey}
of convex sets $K$ and $L$ with the origin in their interior is given
by
\begin{equation*}
  h_{K+_p L}^p(x) = h_K^p(x) + h_L^p(x).
\end{equation*}
The $L_p$-Brunn-Minkowski inequality of Firey states that
\begin{equation}
  \label{eqn:LpFirey}
V_n(K+_p L)^{1/n} \gr V_n(K)^{1/n}+ V_n(L)^{1/n}.
\end{equation}
A more recent pointwise definition that applies to compact sets $K$
and $L$ is due to Lutwak, Yang and Zhang \cite{LYZ_Firey}
\begin{equation}
  \label{eqn:Lp_pointwise}
  K+_p L = \{(1-t)^{1/q}+t^{1/q}y:x\in K, y\in L, 0\ls t\ls 1\},
\end{equation}
where $1/p+1/q=1$; they proved that with the latter definition
\eqref{eqn:LpFirey} extends to compact sets.

A general framework incorporating the latter as well as more general
notions in the Orlicz setting initiated by Lutwak, Yang and Zhang
\cite{LYZ_Orlicz}, \cite{LYZ_Orlicz_projection}, was studied by
Gardner, Hug and Weil \cite{GHW_JEMS}, \cite{GHW_Orlicz}.  Let $M$ be
an arbitrary subset of $\R^m$ and define the {\it $M$-combination}
$\oplus_M (K^1,\ldots,K^m)$ of arbitrary sets $K^1,\ldots, K^m$ in
$\R^n$ by
\begin{eqnarray*}
\oplus_M(K^1,\ldots,K^m) &=&\left\{\sum_{i=1}^m a_i x^{(i)}:
x^{(i)}\in K^i, (a_1,\ldots,a_m)\in M\right\}\\ & = &
\bigcup_{(a_i)\in M} (a_1 K^1+\ldots+a_m K^m).
\end{eqnarray*}

Gardner, Hug, and Weil \cite{GHW_JEMS} develop a general framework for
addition operations on convex sets which model important features of
the Orlicz-Brunn-Minkowski theory.  The notion of $M$-addition is
closely related to linear images of convex sets in this paper. In
particular, if $C=M$ and $K^1=\{x_1\}, \ldots, K^m = \{x_m\}$, where
$x_1,\ldots,x_m\in \R^n$, then $[x_1,\ldots,x_m]C = \oplus_M(\{x_1\},
\ldots,\{x_m\})$.

As a sample result we mention just the following from \cite{GHW_JEMS} (see
Theorem 6.1 and Corollary 6.4).

\begin{theorem} 
\label{thm:Msum}
Let $M$ be a convex set in $\R^m$, $m\gr 2$.
  \begin{itemize}
  \item[i.] If $M$ is contained in the positive orthant and
    $K^1,\ldots,K^m$ are convex sets in $\R^n$, then
    $\oplus_M(K^1,\ldots,K^m)$ is a convex set.
\item[ii.] If $M$ is $1$-unconditional and $K^1,\ldots,K^m$ are
  origin-symmetric convex sets, then $\oplus_M(K^1,\ldots,K^m)$ is an
  origin symmetric convex set.
  \end{itemize}
\end{theorem}

For several examples we mention the following:

\begin{itemize}
  \item[(i)] If $M=\{(1,1)\}$ and $K^1$ and $K^2$ are convex sets, then
    $K^1\oplus_M K^2 = K_1+K_2$, i.e., $\oplus_M$ is the usual
    Minkowski addition.
    \item[(ii)] If $M = B_q^N$ with $1/p+1/q=1$, and $K^1$ and $K^2$
      are origin symmetric convex bodies, then $K^1\oplus_M K^2 =
      K^1+_p K^2$, i.e., $\oplus_M$ corresonds to $L_p$-addition as in
      \eqref{eqn:Lp_pointwise}.
    \item[(iii)] There is a close connection between Orlicz addition
      as defined in \cite{LYZ_Orlicz}, \cite{LYZ_Orlicz_projection}
      and $M$-addition, as shown in \cite{GHW_Orlicz}. In fact, we
      define Orlicz addition in terms of the latter as it interfaces
      well with our operator approach. As an example, let
      $\psi:[0,\infty)^2\rightarrow [0,\infty)$ be convex, increasing
          in each argument, and $\psi(0,0)=0$,
          $\psi(1,0)=\psi(0,1)=1$. Let $K$ and $L$ be origin-symmetric
          convex bodies and let $M = B_{\psi}^{\circ}$, where
          $B_{\psi}=\{(t_1,t_2)\in[-1,1]^2:\psi(\abs{t_1},\abs{t_2})\ls
          1\}$.  Then we define $K+_{\psi} L$ to be $K\oplus_M L$.
\end{itemize}

Let $N_1, \ldots,N_m$ be positive integers.  For each $i=1,\ldots,m$,
consider collections of vectors $\{x_{i1},\ldots,x_{iN_i}\}\subseteq
\R^n$ and let $C_1,\ldots,C_m$ be compact, convex sets with $C_i \subseteq
\R^{N_i}$. Then for any $M\subseteq \R^{N_1+\ldots+N_m}$,
\begin{eqnarray*}
  \lefteqn{\oplus_M([x_{11},\ldots,x_{1N_1}]C_1, \ldots,
    [x_{m1},\ldots,x_{mN_m}]C_m)}\\ & & = \left\{\sum_{i=1}^m a_i
  \left(\sum_{j=1}^{N_i} c_{ij}x_{ij}\right): (a_i)_i\in M,
  (c_{ij})_j\in C_i\right\}\\ & & = \left\{\sum_{i=1}^m\sum_{j=1}^{N_i}
  a_i c_{ij}x_{ij}: (a_i)_i\in M, (c_{ij})_j\in C_i\right\}\\ & & =
  [x_{11},\ldots,x_{1N_1},\ldots,x_{m1},\ldots,x_{mN_m}]
  (\oplus_M(C_1',\ldots,C_m')),
\end{eqnarray*}
where $C_i'$ is the natural embedding of $C_i$ into
$\R^{N_1+\ldots+N_m}$.  Thus the $M$-combination of families of sets
of the form $[x_{i1},\ldots,x_{iN_i}]C_i$ fits exactly in the
framework considered in this paper.  In particular, if $M$ is compact,
convex and satisfies either of the assumptions of Theorem
\ref{thm:Msum}, then the $j$-th intrinsic volume of the latter set is
a Steiner convex function by Corollary \ref{cor:shadow_one}.

For subsequent reference we note one special case of the preceding
identities.  Let $C_1=\conv\{e_1,\ldots,e_{N_1}\}$ and $C_2
=\conv\{e_1,\ldots,e_{N_2}\}$. Then we identify $C_1$ with
$C'_1=\conv\{e_1,\ldots,e_{N_1}\}$ in $\R^{N_1+N_2}$, $C_2$ with
$C_2' = \conv\{e_{N_1+1},\ldots,e_{N_1+N_2}\}$ in $\R^{N_1+N_2}$.
If $x_1,\ldots,x_{N_1}$, $x_{N_1+1},\ldots,x_{N_1+N_2}\in \R^n$, then
\begin{eqnarray*}
  \lefteqn{\conv\{x_1,\ldots,x_{N_1}\}
    \oplus_{M}\conv\{x_{N_1+1},\ldots,x_{N_1+N_2}\}} \\ & & =
          [x_1,\ldots,x_{N_1}]C_1 \oplus_{M}
          [x_{N_1+1},\ldots,x_{N_1+N_2}]C_2 \\ & & =
          [x_1,\ldots,x_{N_1},x_{N_1+1},\ldots,x_{N_1+N_2}](C_1'\oplus_{M}C_2').
\end{eqnarray*}  This will be used in \S \ref{section:iso}.

\subsection{Unions and intersections of Euclidean balls}

Here we consider Euclidean balls $B(x_i,R)=\{x\in \R^n:|x-x_i|\ls r\}$
of a given radius $r>0$ with centers $x_1,\ldots,x_N\in
\R^n$. 

\begin{theorem}
  For each $1\ls j \ls n$, the function
  \begin{equation}
    \label{eqn:ballpoly}
    (\R^n)^N \ni (x_1,\ldots,x_N) \mapsto V_j\left(\bigcap_{i=1}^N
    B(x_i,r)\right)
  \end{equation} 
  is Steiner concave. Moreover, it is quasi-concave and even on
  $(\R^{n})^N$.
\end{theorem}

\begin{proof}
  Let $F$ be the function in \eqref{eqn:ballpoly}.  Let ${\bf u}
  =(u_1,\ldots,u_N)\in (\R^n)^N$ and ${\bf v} = (v_1,\ldots,v_N)\in
  (\R^n)^N$ belong to the support of $F$.  One checks the following
  inclusion,
  \begin{eqnarray*}
    \bigcap_{i=1}^N B\left(\frac{u_i+v_i}{2}, r_i\right) \supseteq
    \frac{1}{2}\bigcap_{i=1}^N B(u_i, r_i) +
    \frac{1}{2}\bigcap_{i=1}^N B(v_i, r_i),
  \end{eqnarray*} 
and then applies the concavity of $K\mapsto V_j(K)^{1/j}$, which is a
consequence of the Alexandrov-Fenchel inequalities.
\end{proof}

\begin{remark}
  The latter theorem is also true when $V_j$ is replaced by a function
  which is monotone with respect to inclusion, rotation-invariant and
  quasi-concave with respect to Minkowski addition; see
  \cite{PaoPiv_ball}.
\end{remark}

The latter can be compared with the following result for the convex
hull of unions of Euclidean balls.

\begin{theorem}
The function 
    $$(\R^n)^N \ni (x_1,\ldots,x_N) \mapsto V_j\left(\mathop{\rm
  conv}\left(\bigcup_{i=1}^N B(x_i,r)\right)\right)$$ is Steiner
convex.
\end{theorem}

\begin{proof}Since
\begin{equation*}  
  \mathop{\rm conv}\left(\bigcup_{i=1}^N B(x_i,r)\right) =
  \conv\{x_1,\ldots,x_N\}+B(0,r),
  \end{equation*}
we can apply the same projection argument as in the proof of Corollary
\ref{cor:shadow_one}; see also work of Pfiefer \cite{Pfiefer} for a
direct argument, extending Groemer's approach \cite{Groemer}.
\end{proof}

\subsection{Operator norms}

Steiner convexity is also present for operator norms from an arbitrary
normed space into $\ell_2^n$.

\begin{proposition}
  \label{prop:operator}
  Let $E$ be an $N$-dimensional normed space.  For $x_1,\ldots,x_N\in
  \R^n$, let ${\bf X} = [x_1,\ldots,x_N]$.  Then the operator norm
  \begin{equation}
    \label{eqn:operator_norm}
    (\R^{n})^N \ni {\bf X} \mapsto \norm{{\bf X}:E\rightarrow \ell_2^n}
  \end{equation}is Steiner convex.
\end{proposition}

\begin{proof}
  Denote the map in \eqref{eqn:operator_norm} by $G$. Then $G$ is
  convex and hence the restriction to any line is convex. In
  particular, if $z \in S^{n-1}$ and $y_1,\ldots,y_N \in z^{\perp}$,
  then the function $G_Y:\R^N\rightarrow \R^{+}$ defined by
    \begin{equation*}
      G_Y(t_1,\ldots,t_N)=G(y_1+t_1z_1,\ldots,y_N+t_Nz_N)
    \end{equation*} is convex.
    To show that $G_Y$ is even, we use the fact that
    $y_1,\ldots,y_N\in z^{\perp}$ to get for any $\lambda\in \R^N$, 
    \begin{eqnarray*}
      \norm{\sum\lambda_i (y_i+t_iz) }_2^2 & = &
      \norm{\sum\lambda_i (y_i-t_iz) }_2^2,
    \end{eqnarray*}hence $G_Y(t)=G_Y(-t)$.
\end{proof}

\section{Stochastic forms of isoperimetric inequalities}
\label{section:iso}

We now have all the tools to prove the randomized inequalities
mentioned in the introduction and others.  We will first prove two
general theorems on stochastic dominance and then show how these imply
a variety of randomized inequalities.  At the end of the section, we
discuss some examples of a different flavor.

For the next two theorems, we assume we have the following sequences
of {\it independent} random vectors defined on a common probability space
$(\Omega, \mathcal{F}, \mathbb{P})$; recall that $\wtd{B} =
\omega_n^{-1/n}B$.

\begin{itemize}
  \item[1.] $X_1,X_2,\ldots$, sampled according to densities
    $f_1,f_2,\ldots$ on $\R^n$, respectively (which will be chosen
    accordingly to the functional under consideration).
  \item[2.] $X^*_1,X^*_2,\ldots$, sampled according to
    $f^*_1,f^*_2,\ldots$, respectively.
  \item[3.] $Z_1,Z_2\ldots$ sampled uniformly in $\wtd{B}$.
\end{itemize}

We use ${\bf X}$ to denote the $n\times N$ random matrix ${\bf
  X}=[X_1\ldots X_N]$. Similarly, ${\bf X^*}=[X_1^*\ldots X_N^*]$ and
${\bf Z} =[Z_1\ldots Z_N]$.

\begin{theorem} 
  \label{thm:main_1}
Let $C$ be a compact convex set in $\R^N$ and $1\ls j\ls n$.  Then for
each $\alpha\gr 0$,
\begin{equation}
  \label{eqn:main_1_a}
  \mathbb{P}(V_j({\bf X}C) > \alpha) \gr \mathbb{P}( V_j({\bf X^*}C)> \alpha).
\end{equation}
Moreover, if $C$ is $1$-unconditional and $\norm{f_i}_{\infty}\ls 1$
for $i=1,\ldots,N$, then for each $\alpha\gr 0$,
\begin{equation}
  \label{eqn:main_1_b}
  \mathbb{P}(V_j({\bf X}C) > \alpha) \gr \mathbb{P}( V_j({\bf Z}C)> \alpha).
\end{equation}
\end{theorem}

  \begin{proof} 
    By Corollary \ref{cor:shadow_one}, we have Steiner convexity. Thus
    we may apply Corollary \ref{cor-multi-BLL} to obtain
    \eqref{eqn:main_1_a}.  If $C$ is unconditional, then Proposition
    \ref{pass-ball-1} applies so we can conclude \eqref{eqn:main_1_b}.
  \end{proof}

\begin{theorem} 
  \label{thm:main_2}
  Let $C$ be an origin symmetric convex body in $\R^N$.  Let $\nu$ be
  a radial measure on $\R^n$ with a density $\psi$ which is
  $-1/(n+1)$-concave on $\R^n$.  Then for each $\alpha\gr 0$,
  \begin{equation}
    \label{eqn:main_2_a}
    \mathbb{P}(\nu(({\bf X}C)^{\circ}) >\alpha) \ls \mathbb{P}(\nu({\bf
      X^*}C)^{\circ}> \alpha).
  \end{equation}
  Moreover, if $C$ is $1$-unconditional and $\norm{f_i}_{\infty}\ls 1$
  for $i=1,\ldots,N$, then for each $\alpha\gr  0$,
  \begin{equation}
    \label{eqn:main_2_b}
    \mathbb{P}(\nu(({\bf X}C)^{\circ}) > \alpha) \ls \mathbb{P}(\nu(({\bf
      Z}C)^{\circ})> \alpha).
  \end{equation}

  \begin{proof} 
    By Corollary \ref{cor:shadow_two}, the function is
    Steiner concave. Thus we may apply Corollary \ref{cor-multi-BLL}
    to obtain \eqref{eqn:main_2_a}.  If $C$ is unconditional, then
    Proposition \ref{pass-ball-1} applies so we can conclude
    \eqref{eqn:main_2_b}.
  \end{proof}
\end{theorem}

We start by explicitly stating some of the results mentioned in the
introduction. We will first derive consequences for points sampled in
convex bodies or compact sets $K\subseteq \R^n$.  In this case, we
have immediate distributional inequalities as
$(\frac{1}{V_n(K)}\mathds{1}_{K})^* =
\frac{1}{V_n(r_KB)}\mathds{1}_{r_KB}$, even without the
unconditionality assumption on $C$. The case of compact sets deserves
special mention for comparison to classical inequalities.  \\

{\noindent \bf 1. Busemann random simplex inequality.} As mentioned
the Busemann random simplex inequality says that if $K\subseteq \R^n$
is a compact set with $V_n(K)>0$ and
$K_{o,n}=\conv\{o,X_1,\ldots,X_n\},$ where $X_1,\ldots,X_n$ are
i.i.d. random vectors with density $f_i=\frac{1}{V_n(K)}\mathds{1}_K$,
then for $p\gr 1$,
\begin{equation}
 \label{eqn:Busemann_body} 
 \EE V_n(K_{o,n})^p \gr \EE V_n((r_KB)_{o,n})^p.
\end{equation}
In our notation, $X_1^*, \ldots,X_n^*$ have density
$\frac{1}{V_n(r_KB)}\mathds{1}_{r_KB}$.  For the set $C =
\conv\{o,e_1,\ldots,e_n\}$, we have $K_{n,o} =
\conv\{o,X_1,\ldots,X_n\}$. Thus the stochastic dominance of Theorem
\ref{thm:main_1} implies \eqref{eqn:Busemann_body} for all $p>0$.\\

{\noindent \bf 2. Groemer's inequality for random polytopes.}  With
the $X_i$'s as in the previous example, set
$K_N=\conv\{X_1,\ldots,X_N\}$. An inequality of Groemer \cite{Groemer}
states that for $p\gr 1$,
\begin{equation}
  \label{eqn:Groemer}
  \EE V_n(K_N)^p \gr \EE V_n((r_KB)_N)^p;
\end{equation}
this was extended by Giannopoulos and Tsolomitis for $p\in (0,1)$ in
\cite{GT_radius}.  Let $C= \conv\{e_1,\ldots,e_N\}$ so that $K_N =
     [X_1,\ldots,X_N]C$ and $(r_KB)_N = [X_1^*,\ldots,X_N^*]C$. Then
     \eqref{eqn:Groemer} follows from Theorem \ref{thm:main_1}.\\

{\noindent \bf 3. Bourgain-Meyer-Milman-Pajor inequality for random
  zonotopes.}  Let $Z_{1,N}(K) =\sum_{i=1}^N [-X_i,X_i]$, with $X_i$
as above. Bourgain, Meyer, Milman, and Pajor \cite{BMMP} proved that
for $p>0$,
\begin{equation}
  \label{eqn:BMMP}
  \EE V_n(Z_{1,N}(K))^p \gr \EE V_n(Z_{1,N}(r_KB))^p.
\end{equation} With the notation of the previous examples, $Z_{1,N}(K) =
[X_1,\ldots,X_N] B_{\infty}^N.$ Thus Theorem \ref{thm:main_1} implies
\eqref{eqn:BMMP}.\\

{\noindent \bf 4. Inequalities for intrinsic volumes.}  For
completeness, we record here how one obtains the stochastic form of
the isoperimetric inequality \eqref{eqn:iso_sd}.  In fact, we state a
stochastic form of the following extended isoperimetric inequality for
convex bodies $K\subseteq \R^n$: for $1\ls j\ls n$,
\begin{equation}
  \label{eqn:ext_iso}
  V_j(K)\gr V_j(r_KB).
\end{equation}
The latter is a consequence of the Alexandrov-Fenchel inequalities,
e.g., \cite{Schneider_book_ed2}.  With $K_N$ as above, a stochastic
form \eqref{eqn:ext_iso} is the following: for $\alpha\gr 0$,
\begin{equation}
  \label{eqn:ext_iso_sd}
\mathbb{P}(V_j(K_{N})> \alpha) \gr \mathbb{P}(V_j((r_KB)_{N})> \alpha),
\end{equation}
which is immediate from Theorem \ref{thm:main_1}. For expectations,
results of this type for intrinsic volumes were proved by Pfiefer
\cite{Pfiefer} and Hartzoulaki and the first named author \cite{HP}.\\

For further information on the previous inequalities and others we
refer the reader to the paper of Campi and Gronchi \cite{CG} and the
references therein. We have singled out these four as particular
examples of $M$-additions (defined in the previous section).  For
example, if $C=\conv\{e_1,\ldots,e_N\}$, we have $$K_{N}
=\oplus_C(\{X_1\},\ldots,\{X_N\}).$$ Similarly, for $C=B_{\infty}^N$,
$$\sum_{i=1}^N[-X_i,X_i]=\oplus_C([-X_1,X_1], \ldots,[-X_N,X_N]).$$ One can also intertwine
the above operations and others.  For example, if
$C=\conv\{e_1,e_1+e_2,e_1+e_2-e_3\}$. Then
\begin{equation*}
[X_1,X_2,X_3]C = \conv\{X_1,X_1+X_2,X_1+X_2-X_3\} 
\end{equation*} and Theorem \ref{thm:main_1} applies 
to such sets as well. The randomized Brunn-Minkowski inequality
\eqref{eqn:BM_sd} is just one example of mixing two operations -
convex hull and Minkowski summation. In the next example, we state a
sample stochastic form of the Brunn-Minkowski inequality for
$M$-addition in which \eqref{eqn:BM_sd} is just a special case; all of
the previous examples also fit in this framework for additional
summands. For other Brunn-Minkowski type inequalities for
$M$-addition, see \cite{GHW_JEMS}, \cite{GHW_Orlicz}.\\

{\noindent \bf 5. Brunn-Minkowski type inequalities.} Let $K$ and $L$
be convex bodies in $\R^n$ and let $M\subseteq \R^2$ be compact,
convex and contained in the positive orthant. Then the following
Brunn-Minkowski type inequality holds for each $1\ls j\ls n$,
\begin{equation}
  \label{eqn:BM_Madd}
  V_j(K\oplus_M L) \gr V_j(r_KB \oplus_M r_LB).
\end{equation}
We first state a stochastic form of the latter. Let $K_{N_1} =
\conv\{X_1,\ldots,X_{N_1}\},$ where $X_1,\ldots,X_{N_1}$ have density
$f_i=\frac{1}{V_n(K)}\mathds{1}_K$; similarly, we define
$L_{N_2}=\conv\{X_{N_1+1},\ldots,X_{N_1+N_2}\},$ where $X_{N_1+1}$,
$\ldots$, $X_{N_1+N_2}$ have density
$f_i=\frac{1}{V_n(L)}\mathds{1}_{L}$.  Then for $\alpha >0$,
\begin{equation}
  \label{eqn:BM_Madd_sd} \mathbb{P}\left(V_j(K_{N_1} \oplus_M
  L_{N_2})> \alpha \right)\gr \mathbb{P}\left(V_j((r_K B)_{N_1} \oplus_M
  (r_LB)_{N_2})> \alpha \right).
\end{equation} 
To see that \eqref{eqn:BM_Madd_sd} holds, set
\begin{equation*}
C_1=\conv\{e_1,\ldots,e_{N_1}\}, \quad
C_2=\conv\{e_1,\ldots,e_{N_2}\}.
\end{equation*}
Identifying $C_1$ with $C'_1=\conv\{e_1,\ldots,e_{N_1}\}$ in
$\R^{N_1+N_2}$ and similarly $C_2$ with $C_2' =
\conv\{e_{N_1+1},\ldots,e_{N_1+N_2}\}$ in $\R^{N_1+N_2}$ as in \S
\ref{sub:Madd}, we have
\begin{eqnarray*}
   K_{N_1}\oplus_M L_{N_2} & = & [X_1,\ldots,X_{N_1}]C_1
     \oplus_M [X_{N_1+1},\ldots,X_{N_1+N_2}]C_2 \\ & = &
           [X_1,\ldots,X_{N_1},X_{N_1+1},\ldots,X_{N_1+N_2}](C'_1\oplus_M C'_2).
\end{eqnarray*}
Write ${\bf X_1} = [X_1,\ldots,X_{N_1}]$ and ${\bf
  X_2}=[X_{N_1+1},\ldots,X_{N_1+N_2}]$, and ${\bf X_1^*} =
[X_1^*,\ldots,X_{N_1}^*]$ and ${\bf
  X_2^*}=[X_{N_1+1}^*,\ldots,X_{N_1+N_2}^*]$.  In block matrix form,
we have
\begin{eqnarray*}
  K_{N_1}\oplus_M L_{N_2} & = & [{\bf X_1},\; {\bf X_2}](C_1'
  \oplus_M C_2').
\end{eqnarray*}
Similarly,
\begin{eqnarray*}
  (r_KB)_{N_1}\oplus_M (r_LB)_{N_2} = [{\bf X_1^*},\; {\bf
      X_2^*}](C_1' \oplus_M C_2'),
\end{eqnarray*} 
and so Theorem \ref{thm:main_1} implies \eqref{eqn:BM_Madd_sd}.  To
prove \eqref{eqn:BM_sd}, we take $M=\{(1,1)\}$ and $j=n$ in
\eqref{eqn:BM_Madd_sd}.  Inequality \eqref{eqn:BM_Madd} follows from
\eqref{eqn:BM_Madd_sd} when $N_1,N_2\rightarrow \infty$. For
simplicity of notation, we have stated this for only two sets and
$C_1, C_2$ as above.  \\

For another example involving a law of large numbers, we turn to the
following, stated in the symmetric case for simplicity.\\

{\noindent \bf 6. Orlicz-Busemann-Petty centroid inequality.} Let
$\psi:[0,\infty)\rightarrow [0,\infty)$ be a Young function, i.e.,
    convex, strictly increasing with $\psi(0)=0$. Let $f$ be a bounded
    probability density of a continuous distribution on $\R^n$. Define
    the Orlicz-centroid body $Z_{\psi}(f)$ associated to $\psi$ by its
    support function
  \begin{equation*}
    h(Z_{\psi}(f),y)=\inf\left\{
      \lambda >0:
      \int_{\R^n}\psi\left(
        \frac{\abs{\langle x, y \rangle}}{\lambda}
      \right)f(x) dx 
      \ls 1 \right\}.
  \end{equation*}    
  Let $r_f>0$ be such that $\norm{f}_{\infty} \mathds{1}_{r_f B}$ is a
  probability density. Then
  \begin{equation}
    \label{eqn:Orlicz_centroid}
   V_n(Z_{\psi}(f))\gr V_n(Z_{\psi}(\norm{f}_{\infty} \mathds{1}_{r_f
     B}).
 \end{equation}
Here we assume that $h(Z_{\psi}(f),y)$ is finite for each $y\in
S^{n-1}$ and so $h(Z_{\psi}(f),\cdot)$ defines a norm and hence is the
support function of the symmetric convex body $Z_{\psi}(f)$. When $f$
is the indicator of a convex body, \eqref{eqn:Orlicz_centroid} was
proved by Lutwak, Yang and Zhang \cite{LYZ_Orlicz} (where it was also
studied for more general functions $\psi$); it was extended to star
bodies by Zhu \cite{Zhu}; the version for probability densities and
the randomized version below is from \cite{PaoPiv_probtake}; an
extension of \eqref{eqn:Orlicz_centroid} to the asymmetric case was
carried out by Huang and He \cite{HuangHe}.

The empirical analogue of \eqref{eqn:Orlicz_centroid} arises by
considering the following finite-dimensional origin-symmetric Orlicz
balls
\begin{equation*}
  B_{\psi, N}:=\left\{ t=(t_1,\ldots,t_N)\in \R^N:
  \frac{1}{N}\sum_{i=1}^N\psi(\abs{t_i})\ls 1\right\}
\end{equation*}with associated Orlicz norm
$\norm{t}_{B_{\psi/N}}:=\inf\{\lambda >0: t\in \lambda B_{\psi, N}\},$
which is the support function for $B_{\psi, N}^{\circ}$. For
independent random vectors $X_1,\ldots,X_N$ distributed according to
$f$, we let
$$Z_{\psi, N}(f) = [X_1,\ldots, X_N]B_{\psi, N}^{\circ}.$$ Then for
$y\in S^{n-1}$,
\begin{equation*}
   h(Z_{\psi, N}(f), y) = \norm{(\langle X_1,y\rangle,\ldots,\langle
     X_N,y\rangle)}_{B_{\psi/N}}.
\end{equation*}
Applying Theorem \ref{thm:main_1} for $C= B^{\circ}_{\psi,N}$, we get that for
$1\ls j \ls n$ and $\alpha\gr 0$,
\begin{equation}
  \label{eqn:Orlicz_centroid_emp}
  \mathbb{P}(V_j(Z_{\psi,N}(f))> \alpha)\gr
  \mathbb{P}(V_j(Z_{\psi,N}(\norm{f}_{\infty}\mathds{1}_{r_f B}))> \alpha).
\end{equation}
Using the law of large numbers, one may check that 
\begin{equation}
  \label{eqn:psi_converge}
  Z_{\psi,N}(f) \rightarrow Z_{\psi}(f)
\end{equation}
almost surely in the Hausdorff metric (see \cite{PaoPiv_probtake});
when $\psi(x) =x^p$ and $f=\frac{1}{V_n(K)}\mathds{1}_K$,
$Z_{\psi,N}(f) = Z_{p,N}(K)$ as defined in the introduction; in this
case, the convergence in \eqref{eqn:psi_converge} is immediate by the
classical strong law of large numbers (compare \eqref{eqn:Zp} and
\eqref{eqn:Zp_empirical}). By integrating
\eqref{eqn:Orlicz_centroid_emp} and sending $N\rightarrow \infty$, we
thus obtain \eqref{eqn:Orlicz_centroid}. \\

We now turn to the dual setting.\\

{\noindent \bf 7.  Blaschke-Santal\'{o} type inequalities.} The
Blaschke-Santal\'{o} inequality states that if $K$ is a symmetric
convex body in $\R^n$, then \begin{equation}
  \label{eqn:BS_body}
  V_n(K^{\circ})\ls V_n((r_KB)^{\circ}).
\end{equation}
This was proved by Blaschke for $n=2,3$ and in general by Santal\'{o}
\cite{Santalo}; see also Meyer and Pajor's proof by Steiner
symmetrization \cite{MeyerPajor} and \cite{Schneider_book_ed2},
\cite{Gardner_book} for further background; origin symmetry in
\eqref{eqn:BS_body} is not needed but we discuss the randomized
version only in the symmetric case.  One can obtain companion results
for all of the inequalities mentioned so far with suitable choices of
symmetric convex bodies $C$.  Let $\nu$ be a radially decreasing
measure as in Theorem \ref{thm:main_2}. Let $C= B_1^N$ and set
$K_{N,s} =[X_1,\ldots,X_N] B_1^N$, where $X_i$ has density
$f_i=\frac{1}{V_n(K)}\mathds{1}_{K}$. Then for $\alpha >0$,
\begin{equation*}
   \mathbb{P}(\nu((K_{N,s})^{\circ})> \alpha) 
   \ls \mathbb{P}(\nu(((r_KB)_{N,s})^{\circ})> \alpha).
\end{equation*}
Similarly, if $K$ and $L$ are origin-symmetric convex bodies and
$M\subseteq \R^2$ is unconditional, then for $\alpha >0$,
\begin{equation}
  \mathbb{P}(\nu((K_{N_1,s}\oplus_M L_{N_1,s})^{\circ})> \alpha)\ls
  \mathbb{P}(\nu(((r_KB)_{N_1,s}\oplus_M (r_LB)_{N_1,s})^{\circ})> \alpha).
\end{equation}

We also single out the polar dual of the last example on
Orlicz-Busemann-Petty centroid bodies.  Let $\psi$ and $B_{\psi,N}$ be
as above. Then
\begin{equation*}
  \mathbb{P}(\nu(Z^{\circ}_{\psi,N}(f)) > \alpha) \ls
  \mathbb{P}(\nu(Z^{\circ}_{\psi,N}(\norm{f}_{\infty}\mathds{1}_{r_f
    B})) > \alpha).
\end{equation*}
For a particular choice of $\psi$ we arrive at the following example,
which has not appeared in the literature before and deserves an
explicit mention. \\

{\noindent \bf 8. Level sets of the logarithmic Laplace transform.}
For a continuous probability distribution with an even bounded density
$f$, recall that the logarithmic Laplace transform is defined by
\begin{equation*}
  \Lambda(f,y) = \log \int_{\R^n} \exp\left(\langle x,y\rangle
  \right)f(x)dx.
\end{equation*}
For such $f$ and $p>0$, we define an origin-symmetric convex body
$\Lambda_{p}(f)$ by
\begin{equation*}
  \Lambda_{p}(f) = \{y\in \R^n: \Lambda_{f}(y)\ls p\}.
\end{equation*}
The empirical analogue is defined as follows: for independent random
vectors $X_1,\ldots,X_N$ with density $f$, set
\begin{equation*}
  \Lambda_{p,N}(f) = \left\{y\in \R^n: \frac{1}{N}\sum_{i=1}^N
  \psi(\abs{\langle X_i, y \rangle})\ls e^p\right\}.
\end{equation*}
If we set $\psi_p(x) = e^{-p}(e^x-1)$ then
$([X_1,\ldots,X_N]B^{\circ}_{\psi_p,N})^{\circ} = \Lambda_{p,N}(f)$.
Then we have the following stochastic dominance
\begin{equation*}
  \mathbb{P}(\nu(\Lambda_{p,N}(f))> \alpha )\ls 
  \mathbb{P}(\nu(\Lambda_{p,N}(\norm{f}_{\infty}\mathds{1}_{r_f B}))> \alpha ),
\end{equation*}where $r_f$ satisfies $\norm{f}_{\infty} \mathds{1}_{r_f B}=1$.
When $N\rightarrow\infty$, we get
\begin{equation*}
  \nu(\Lambda_{p}(f))\ls \nu(\Lambda_{p}(\norm{f}_{\infty}
  \mathds{1}_{r_fB}).
\end{equation*}
The latter follows from the law of large numbers as in \cite[Lemma
  5.4]{PaoPiv_probtake} and the argument given in \cite[\S 5]{CEFPP}.

For $\log$-concave densities, the level sets of the logarithmic
Laplace transform are known to be isomorphic to the duals to the
$L_p$-centroid bodies; see work of Lata{\l}a and Wojtaszczyk
\cite{LW_inf}, or Klartag and E. Milman \cite{KM_unified}; these bodies
are essential in establishing concentration properties of
$\log$-concave measures, e.g., \cite{Paouris_GAFA},
\cite{Klartag_CLT}, \cite{BGVV}.\\

{\noindent \bf 9. Ball-polyhedra.} All of the above inequalities are
volumetric in nature. For convex bodies, they all reduce to
comparisons of bodies of a given volume. For an example of a different
flavor, we have the following inequality involving random ball
polyhedra: for $R>0$,
\begin{equation*}
  \mathbb{P}\left(V_j\left(\bigcap\nolimits_{i=1}^N B(X_i,R)\right)\gr
  \alpha\right) \ls \mathbb{P}\left(V_j\left(\bigcap\nolimits_{i=1}^N
  B(Z_i,R)\right)> \alpha\right).
\end{equation*}
When the $X_i$'s are sampled according to a particular density $f$
associated with a convex body $K$, the latter leads to the following
generalized Urysohn inequality,
\begin{equation*}
  V_j(K)\ls V_j((w(K))/2)B),
\end{equation*}
where $w(K)$ is the mean width of $K$, see \cite{PaoPiv_ball}; the
latter is not a volumetric inequality when $j<n$.  The particular
density $f$ is the uniform measure on a star-shaped set $A(K,R)$
defined by specifying its radial function $\rho_{A(K,R)}(\theta) =
R-h_K(-\theta)$; Steiner symmetrization of $A(K,R)$ preserves the
mean-width of $K$ (for large $R$) so the volumetric techniques here
lead to a stochastic dominance inequality for mean width.\\

We have focused this discussion on stochastic dominance.  It is
sometimes useful to relax the probabilistic formulation and instead
consider the quantities above in terms of bounded integrable
functions. We give one such example. \\

{\bf 10. Functional forms.} The following functional version of Busemann's
random simplex inequality \eqref{eqn:Busemann} is useful for marginal
distributions of high-dimensional probability distributions; this is
from joint work with S. Dann \cite{DPP}. Let $f_1,\ldots,f_k$ be
non-negative, bounded, integrable functions such that $\norm{f_i}_1>0$
for each $i=1,\ldots,k$. For $p\in \R$, set
  \begin{equation*}
    g_{p}(f_1,\ldots,f_k) =
    \int_{\R^n}\cdots\int_{\R^n}V_k(\conv\{0,x_1,\ldots,x_k\})^p
    \prod_{i=1}^k f_i(x_i)dx_1\ldots dx_k.
  \end{equation*}
  Then for $p>0$,
  \begin{equation*}
    g_{p}(f_1,\ldots,f_k) \gr
    \left(\prod_{i=1}^k\frac{\norm{f_i}_1^{1+p/n}}
         {\omega_{n}^{1+p/n}\norm{f_i}_{\infty}^{p/n}}\right)
         g_{p}(\mathds{1}_{B_2^n},\ldots,\mathds{1}_{B_2^n}).
   \end{equation*}
The latter is just a special case of a general functional inequality
\cite{DPP}.  Following Busemann's argument, we obtain the
following. Let $1\ls k\ls n-1$ and let $f$ be a non-negative, bounded
integrable function on $\R^n$.  Then
  \begin{equation*}
    \int_{G_{n,k}}
    \frac{\left(\int_{E}f(x)dx\right)^{n}}{\norm{f\lvert_E}_{\infty}^{n-k}}
    d\nu_{n,k}(E) \ls \frac{\omega_{k}^n}{\omega_n^{k}}
    \left(\int_{\R^n}f(x)dx\right)^k;
  \end{equation*} when $f= \mathds{1}_K$ this recovers the 
  inequality of Busemann and Straus \cite{BusemannStraus} and Grinberg
  \cite{Grinberg} extending \eqref{eqn:Busemann_int}. Schneider proved
  an analogue of the latter on the affine Grassmannian
  \cite{Schneider_flats}, which can also be extended to a sharp
  isoperimetric inequality for integrable functions \cite{DPP}.  The
  functional versions lead to small ball probabilities for projections
  of random vectors that need not have independent coordinates.

\section{An application to operator norms of random matrices}

\label{section:apps}

%\noindent {\bf Applications to random matrices and high dimensional
%  distributions}

In the previous section we gave examples of functionals on random convex
sets which are minorized or majorized for the uniform measure on the
Cartesian product of Euclidean balls.  In some cases the associated
distribution function can be accurately estimated.  For example,
passing to complements in \eqref{eqn:main_1_b}, we get for
$\alpha\gr 0$,
\begin{equation}
  \label{eqn:small_DCG}
  \mathbb{P}(V_n({\bf X}C) \ls \alpha) \ls \mathbb{P}( V_n({\bf
    Z}C)\ls \alpha),
\end{equation}where ${\bf X}$ and ${\bf Z}$ are as in Theorem \ref{thm:main_1}.
When $C=B_1^N$, i.e., for random symmetric convex hulls, we have
estimated the quantity on the righthand side of \eqref{eqn:small_DCG}
in \cite{PaoPiv_smallball} for all $\alpha$ less than an absolute
contant (sufficiently small), at least when $N\ls e^n$.  (The reason
for the restriction is that we compute this for Gaussian matrices and
the comparison to the uniform measure on the Cartesian products of
balls is only valid in this range).  This leads to sharp bounds for
small deviation probabilities for the volume of random polytopes that
were known before only for certain sub-gaussian distributions. The
method of \cite{PaoPiv_smallball} applies more broadly.  In this
section we will focus on the case of the operator norm of a random
matrix with independent columns. We refer readers interested in
background on non-asymptotic random matrix theory to the article of
Rudelson and Vershynin \cite{RV_survey} and the references therein.

By combining Corollary \ref{cor-multi-BLL}, and Propositions
\ref{prop:passball_Kanter} and \ref{prop:operator}, we get the
following result, which is joint work with G. Livshyts
\cite{LPP_progress}.

\begin{theorem}
\label{thm:op_norm_dist}
Let $N, n \in \N$. Let $E$ be an $N$-dimensional normed space. Then
the random matrices ${\bf X}, {\bf X}^{\ast}$ and ${\bf Z}$ (as in \S
\ref{section:iso}) satisfy the following for each $\alpha \gr 0$,
\begin{equation}
 \mathbb P \left( \| {\bf X} : E \rightarrow \ell_{2}^{n} \| \ls
 \alpha \right) \ls \mathbb P \left( \| {\bf X}^{\ast} : E
 \rightarrow \ell_{2}^{2} \| \ls \alpha \right).
\end{equation}
Moreover, if $\norm{f_i}_{\infty} \ls 1$ for each $i=1,\ldots,N$,
then
\begin{equation*}
  \mathbb{P}\left( \| {\bf X} : E \rightarrow \ell_{2}^{n} \| \ls \alpha
 \right) \ls \mathbb{P}\left( \| {\bf Z} : E \rightarrow \ell_{2}^{2}
 \| \ls \alpha \right). 
\end{equation*}
\end{theorem}

As before, the latter result reduces the small deviation problem to
computations for matrices ${\bf Z}$ with independent columns sampled
in the Euclidean ball of volume one. For the important case of the
operator norm $\norm{\cdot}_{2\rightarrow 2}$, i.e.,
$E:=\ell_{2}^{N}$, we get the following bound.

\begin{lemma}
For $\eps>0$,
\begin{equation}
  \mathbb{P}\left(\norm{\bf Z}_{2\rightarrow 2} \ls \varepsilon \sqrt{N}\right)
  \ls (c\varepsilon )^{nN-1}, 
\end{equation}where $c$ is an absolute constant.
\end{lemma}

\begin{proof}
Let $C$ and $K$ be symmetric convex bodies in $\mathbb R^{d}$, $
V_d(K)=1$ and $ p< d $. By \cite[Proposition 4.7]{Paouris_TAMS}), 
\begin{equation}
 \left( \int_{K} \|x\|_{C}^{-p} d x \right)^{\frac{1}{p}} \leq \left(
 \frac{d}{d-p}\right)^{\frac{1}{p}} V_d(C)^{\frac{1}{
     d}}.\end{equation} Let $d:= nN$, $K:= \wtd{B} \times \cdots
\times \wtd{B} \subseteq \mathbb R^{d}$ and $C$ be the unit ball in
$\R^d$ for the operator norm $\norm{\cdot:\ell_2^N\rightarrow
  \ell_2^n}$.  Then the Hilbert-Schmidt norm $\norm{\cdot}_{HS}$
satisfies $\norm{A}_{HS} \leq \sqrt{n} \norm{A}_{2\rightarrow 2}$ or
$C\subseteq \sqrt{n} B_{2}^{d} $, which implies that
$V_d(C)^{\frac{1}{d}} \leq \frac{c_1}{\sqrt{N}}$; in fact, arguing as
in \cite[Lemma 38.5]{NTJ} one can show that $V_d(C)^{\frac{1}{
    d}}\simeq \frac{1}{\sqrt{N}}$. Thus for $ p= nN-1$, we get
\begin{equation*}
  \left(\EE\norm{{\bf Z}}_{2\rightarrow
    2}^{-(nN-1)}\right)^{\frac{1}{nN-1}} \ls c_1(nN)^{\frac{1}{nN-1}}
  N^{-1/2}\ls ec_1 N^{-1/2},
\end{equation*}from which the lemma follows by an application of  Markov's inequality.
\end{proof}

For $1\times N$ matrices, Theorem \ref{thm:op_norm_dist} reduces to
small-ball probabilities for norms of a random vector $x$ in $\R^N$
distributed according to a density of the form $\prod_{i=1}^N f_i$
where each $f_i$ is a density on the real line. In particular, if
$\norm{f_i}_{\infty}\leq 1$ for each $i=1,\ldots,N$, then for any norm
$\norm{\cdot}$ on $\R^N$ (dual to $E$), we have for $\eps>0$,
\begin{equation}
  \mathbb{P}\left(\norm{x}\ls \eps\right) \ls
  \mathbb{P}\left(\norm{z}\ls \eps\right),
\end{equation}where $z$ is a random vector in the cube $[-1/2,1/2]^N$
- the uniform measure on Cartesian products of ``balls'' in
$1$-dimension.  In fact, by approximation from within, the same result
holds if $\norm{\cdot}$ is a semi-norm. Thus if $x$ and $z$ are as
above, for each $\eps>0$ we have
\begin{equation}
  \label{eqn:RV_sb}
  \mathbb{P}(\norm{P_Ex}_2\ls \eps \sqrt{k})\ls
  \mathbb{P}(\norm{P_Ez}_2\ls \eps \sqrt{k}) \ls (2\sqrt{\pi
    e}\eps)^k,
\end{equation} 
where the last inequality uses a result of Ball \cite{Ball_GAFA}.  In
this way we recover the result of Rudelson and Vershynin from
\cite{RV_IMRN}, who proved \eqref{eqn:RV_sb} with a bound of the form
$(c\eps)^k$ for some absolute constant $c$. Using the
Rogers/Brascamp-Lieb-Luttinger inequality and Kanter's theorem, one
can also obtain the sharp constant of $\sqrt{2}$ for the
$\ell_{\infty}$-norm of marginal densities, which was first computed
in \cite{LPP} by adapting Ball's arguments from \cite{Ball_GAFA}.

\subsection*{Acknowledgements}

It is our pleasure to thank Petros Valettas for useful discussions.
We also thank Beatrice-Helen Vritsiou for helpful comments on an
earlier draft of this paper. Part of this work was carried out during
the Oberwolfach workshops {\it Convex Geometry and its Applications},
held December 6-12, 2015, and {\it Asymptotic Geometric Analysis},
held February 21-27, 2016.  We thank the organizers of these meetings
and the institute staff for their warm hospitality.

\bibliographystyle{amsplain} \bibliography{surveybib}

\end{document}